\documentclass[reqno,11pt]{amsart}
\usepackage[utf8]{inputenc}
\usepackage{amsmath, latexsym, amsfonts, amssymb, amsthm, amscd, bm}
 \usepackage{graphics,epsf,psfrag}
\usepackage[dvipsnames]{xcolor}
\usepackage{mathrsfs}
\usepackage{url}

\numberwithin{equation}{section}

\newtheorem{theorem}{Theorem}[section]
\newtheorem{lemma}[theorem]{Lemma}
\newtheorem{proposition}[theorem]{Proposition}

\newtheorem{rem}[theorem]{Remark}

\renewcommand{\tilde}{\widetilde}
\xdefinecolor{red}{named}{Red}

\newcommand{\cB}{{\ensuremath{\mathcal B}} }

\newcommand{\cP}{{\ensuremath{\mathcal P}} }

\newcommand{\cC}{{\ensuremath{\mathcal C}} }

\newcommand{\cL}{{\ensuremath{\mathcal L}} }

\newcommand{\cV}{{\ensuremath{\mathcal V}} }

\newcommand{\cM}{{\ensuremath{\mathcal M}} }
\newcommand{\cW}{{\ensuremath{\mathcal W}} }

\DeclareMathSymbol{\leqslant}{\mathalpha}{AMSa}{"36} % nicer `smaller or equal'
\DeclareMathSymbol{\geqslant}{\mathalpha}{AMSa}{"3E} % nicer `larger or equal'
\DeclareMathSymbol{\eset}{\mathalpha}{AMSb}{"3F}     % nicer `emptyset'
\renewcommand{\leq}{\;\leqslant\;}                   % redef. of < or =
\renewcommand{\geq}{\;\geqslant\;}                   % redef. of > or =
\newcommand{\dd}{\,\text{\rm d}}             % a straight d for differentials
\newcommand{\bbC}{{\ensuremath{\mathbb C}} }

\newcommand{\bbE}{{\ensuremath{\mathbb E}} }

\newcommand{\bbN}{{\ensuremath{\mathbb N}} }

\newcommand{\bbR}{{\ensuremath{\mathbb R}} }

\newcommand{\bbZ}{{\ensuremath{\mathbb Z}} }

\newcommand{\ga}{\alpha}

\newcommand{\gd}{\delta}
\newcommand{\gep}{\varepsilon}       % \ge already exists...

\newcommand{\gs}{\sigma}

\makeatletter
\def\captionfont@{\footnotesize}
\def\captionheadfont@{\scshape}

\long\def\@makecaption#1#2{%
  \vspace{2mm}
  \setbox\@tempboxa\vbox{\color@setgroup
    \advance\hsize-6pc\noindent
    \captionfont@\captionheadfont@#1\@xp\@ifnotempty\@xp
        {\@cdr#2\@nil}{.\captionfont@\upshape\enspace#2}%
    \unskip\kern-6pc\par
    \global\setbox\@ne\lastbox\color@endgroup}%
  \ifhbox\@ne % the normal case
    \setbox\@ne\hbox{\unhbox\@ne\unskip\unskip\unpenalty\unkern}%
  \fi
  \ifdim\wd\@tempboxa=\z@ % this means caption will fit on one line
    \setbox\@ne\hbox to\columnwidth{\hss\kern-6pc\box\@ne\hss}%
  \else % tempboxa contained more than one line
    \setbox\@ne\vbox{\unvbox\@tempboxa\parskip\z@skip
        \noindent\unhbox\@ne\advance\hsize-6pc\par}%
\fi
  \ifnum\@tempcnta<64 % if the float IS a figure...
    \addvspace\abovecaptionskip
    \moveright 3pc\box\@ne
  \else % if the float IS NOT a figure...
    \moveright 3pc\box\@ne
    \nobreak
    \vskip\belowcaptionskip
  \fi
\relax
}
\makeatother
\def\writefig#1 #2 #3 {\rlap{\kern #1 truecm
\raise #2 truecm \hbox{#3}}}

\makeatletter
\newsavebox{\@brx}
\newcommand{\llangle}[1][]{\savebox{\@brx}{\(\m@th{#1\langle}\)}%
  \mathopen{\copy\@brx\kern-0.5\wd\@brx\usebox{\@brx}}}
\newcommand{\rrangle}[1][]{\savebox{\@brx}{\(\m@th{#1\rangle}\)}%
  \mathclose{\copy\@brx\kern-0.5\wd\@brx\usebox{\@brx}}}
\makeatother

\newcommand{\Tr}{\mathrm{Tr}}

\title[Periodic solutions for nonlinear Fokker-Planck equations]{Existence, stability and regularity of periodic solutions for nonlinear Fokker-Planck equations}

\author{Eric Lu\c{c}on}
\address{Universit\'e de Paris, Laboratoire MAP5 (UMR CNRS 8145), 75270 Paris, France \& FP2M, CNRS FR 2036, \url{eric.lucon@u-paris.fr}.
}

\author{Christophe Poquet}
\address{Université Claude Bernard Lyon 1, CNRS UMR 5208, Institut Camille Jordan, F-69622 Villeurbanne, France, \url{poquet@math.univ-lyon1.fr}}


\date{\today}

\begin{document}

\maketitle

\begin{abstract}
We consider a class of nonlinear Fokker-Planck equations describing the dynamics of an infinite population of units within mean-field interaction. Relying on a slow-fast viewpoint and on the theory of approximately invariant manifolds we obtain the existence of a stable periodic solution for the PDE, consisting of probability measures. Moreover we establish the existence of a smooth isochron map in the neighborhood of this periodic solution.
\end{abstract}

\section{Introduction}

\subsection{The model}\label{sec:the model}
We are interested in this paper in the existence, stability and regularity of periodic solutions to the following nonlinear PDE on $\bbR^d$ ($d\geq1$):
\begin{equation}\label{eq:PDE non centered}
\partial_t u_t =  \nabla\cdot \left(\gs^2 u_t\right) +\nabla \cdot \left(K\left(x-\int_{\bbR^d}y u_t({\rm d}y)\right) u_t\right)-\gd \nabla\cdot \left( F(x)u_t\right).
\end{equation}
Here, $t\geq0 \mapsto u_{ t}$ is a probability measure-valued process on $ \mathbb{ R}^{ d}$, $K$ and $\gs$ and diagonal matrices with positive coefficients and $F:\bbR^d\rightarrow\bbR^d$ is a smooth bounded function with bounded derivatives. Equation \eqref{eq:PDE non centered} has a natural probabilistic interpretation: if $u_0$ is a probability distribution on $ \mathbb{ R}^{ d}$, it is well known \cite{McKean:1966,sznitman1991topics} that $u_t$ is the law of the McKean-Vlasov process $X_t$ where $X_0\sim u_0$ and
\begin{equation}\label{eq:McKean}
d X_t =\gd F(X_t)\dd t - K\left(X_t-\bbE[X_t]\right) {\rm d}t+\sqrt{2}\gs \dd B_t,
\end{equation}
The dynamics of the process $(X_{ t})_{ t\geq0}$ is the superposition of a local part $ \delta F(X_{ t}) {\rm d}t$, where $ \delta>0$ is a scaling parameter, a linear interaction term $K\left(X_t-\bbE[X_t]\right) {\rm d}t$, modulated by the intensity matrix $K$, and an additive noise given by a standard Brownian motion $(B_t)_{t\geq 0}$ on $\bbR^d$. The difficulty in the analysis of \eqref{eq:McKean} lies in its nonlinear character: $X_{ t}$ interacts with its own law, more precisely its own expectation $ \mathbb{ E} \left[X_{ t}\right]$. The long-time dynamics of \eqref{eq:McKean} is a longstanding issue in the literature. In particular, the existence of stable equilibria for \eqref{eq:PDE non centered} (that is invariant measures for \eqref{eq:McKean}) has been studied for various choices of dynamics, interaction and regimes of parameters $ \delta, K, \sigma$, mostly in a context where the corresponding particle dynamics defined in \eqref{eq:syst part} below is reversible (see e.g. \cite{10.1214/12-AOP749,MR1632197,Cattiaux2008} for further details and references).

The question we address in the present paper concerns the existence of periodic solutions to nonlinear equation such as \eqref{eq:PDE non centered}. In this case, a major difficulty lies in the fact that the underlying microscopic dynamics is not reversible. From a practical perspective, the emergence of periodicity in such models relates in particular to chemical reactions (Brusselator model \cite{scheutzow1986periodic}), neurosciences \cite{Giacomin:2012,giacomin2015noise,22657695,MR3392551,ditlevsen2017multi,Lucon:2018qy, Quininao:2020,2018arXiv181100305L,Cormier2020}, and statistical physics (e.g. spin-flip models \cite{Collet:2015,DaiPra2020}, see also to \cite{Cerf:2020}, where the model considered is in fact not mean-field, but the Ising model with dissipation). An example of particular interest concerns the FitzHugh-Nagumo model \cite{MR1779040,22657695} (take $d=2$ and $F(x,y)=\left(x-\frac{x^3}{3}-y,\frac1c\left(x+a-by \right) \right)$), commonly used as a prototype for  excitability in neuronal models \cite{LINDNER2004321} or in physics \cite{PhysRevE.68.036209}. Roughly speaking, excitability refers to the ability for a neuron to emit spikes (oscillations) in the presence of perturbations (such as noise and/or external input) whereas this neuron would be at rest (steady state) without perturbation. The long-time dynamics of \eqref{eq:PDE non centered} in the FizHugh-Nagumo case has been the subject of several previous works (existence of equilibria \cite{Mischler2016,Quininao:2020} or periodic solutions \cite{Lucon:2018qy,2018arXiv181100305L}) under various asymptotics of the parameters $( \delta, K, \sigma)$. A crucial feature in this context is the influence of noise and interaction in the emergence and stability of periodic solutions: generically, some balance has to be found in the quantity of noise and interaction that one needs to put in the system in order to observe oscillations (see \cite{LINDNER2004321,Lucon:2018qy,2018arXiv181100305L} for further details). 

\subsubsection{Stability properties and regular isochron map}
The purpose of the present paper is to complement the previous results concerning the existence of periodic orbits for \eqref{eq:PDE non centered} with accurate stability properties for this periodic solution and with the existence of a sufficiently regular isochron map, properties that are absent in the previous works cited above. We obtain these additional properties by applying a result concerning normally hyperbolic invariant manifolds in Banach spaces proved by Bates, Lu and Zeng \cite{Bates:2008}. The technical counterpart is that we require assumptions on $F$ and $\gs$ that are somehow stricter than the ones used in \cite{scheutzow1986periodic, Lucon:2018qy, Quininao:2020,2018arXiv181100305L}, in the sense that we are considering a dynamics $F$ that is bounded together with all  its derivatives (the Brusselator model and FitzHugh Nagumo model considered in the references given above have polynomial growth) as well as nondegenerate noise on all components (while in \cite{Quininao:2020,2018arXiv181100305L} the noise is only present in one of the two variables). 

\subsubsection{Large time asymptotics for the mean-field particle system}
Standard propagation of chaos results \cite{sznitman1991topics} shows that \eqref{eq:McKean} is the natural limit of the following mean-field particle system
\begin{equation}\label{eq:syst part}
d X_{i,t} =\gd F(X_{i,t})\dd t - K\left(X_{i,t}-\frac{1}{N}\sum_{j=1}^N X_{j,t}\right)\dd t+\sqrt{2}\gs \dd B_{i,t},
\end{equation}
in the sense that one can easily couple \eqref{eq:syst part} and \eqref{eq:McKean} by choosing the same realization of the noise, so that the resulting error is of order $ \frac{ 1}{ \sqrt{ N}}$ as $N\to\infty$, at least on any $[0,T]$ with $T$ that can be arbitrarily large but fixed independently from $N$. At the level of the whole particle system, this boils down to the convergence as $N\to \infty$ of the empirical measure $u_{N,t} = \frac{1}{N}\sum_{i=1}^N \gd_{X_{i,t}}$ to $u_{ t}$, solution to \eqref{eq:PDE non centered}. Hence, supposing that \eqref{eq:PDE non centered} has a periodic solution $\left(\Gamma^\gd_t\right)_{t\geq 0}$, if the empirical measure $u_{N, 0}$ is initially close to $\Gamma^\gd_{\theta_0}$, $u_{ N, t}$ has, for $N$ large, a behavior close to being periodic, since it stays close to $\Gamma^\gd_{\theta_0+t}$. 

The companion paper \cite{LP2021b} of the present work is concerned with the behavior of the empirical measure $ u_{ N, t}$ on a time scale $T$ that is no longer bounded, but of order $N$. We show in \cite{LP2021b} that $u_{N,Nt}$ is close to $\Gamma^\gd_{\theta_0+Nt+\beta^N_t}$, where $\beta^N_t$ is a random process in $\bbR$ whose weak limit as $N\to \infty$ has constant drift and diffusion coefficient. This kind of result was already obtained in \cite{dahms2002,bertini2014synchronization} in the case of the plane rotators model (mean-field noisy interacting oscillators defined on the circle), for which in the scale $Nt$ the empirical measure has a diffusive behavior along the curve of stationary points. Our aim in \cite{LP2021b} is to get similar results for models like \eqref{eq:PDE non centered} that are defined in $\bbR^d$, and are not reversible (while the plane rotators model is). As we will explain in more details later, the additional stability and regularity results concerning periodic solution to \eqref{eq:PDE non centered} obtained in the present paper are crucial for the study of long time behavior of the mean-field particle systems \eqref{eq:syst part} made in \cite{LP2021b}.

\subsection{Slow fast viewpoint and application to the FitzHugh-Nagumo model}\label{sec:slow fast}
We give in this paragraph informal intuition on the possibility of emergence of periodic solutions to \eqref{eq:PDE non centered}. The point of view we adopt here is a slow-fast approach, based on the assumption that the parameter $ \delta$ in \eqref{eq:PDE non centered} is small, as it was already the case in \cite{Lucon:2018qy, 2018arXiv181100305L}. More precisely, the linear character of the interaction term in \eqref{eq:PDE non centered} allows us to decompose the dynamics of \eqref{eq:PDE non centered} into its expectation $m_t=\int_{\bbR d}  x u_t(x)$ and its centered version $p_t(x)=u_t(x-m_t)$: \eqref{eq:PDE non centered} is equivalent to the system
\begin{equation}\label{eq:syst PDE gd}
\left\{
\begin{array}{rcl}
\partial_t p_t &=& \cL p_t -\nabla\cdot (p_t (\gd F_{m_t}-\dot m_t))\\
\dot m_t &= & \gd \int F_{m_t} \dd p_t 
\end{array}
\right. ,
\end{equation}
where
\begin{equation}
\cL u =\nabla\cdot (\gs^2\nabla f)+ \nabla \cdot\left( Kx  f\right),
\end{equation}
and $F_m(x)=F(x+m)$. Remark that $(p_t,m_t)$ is the weak limit as $N\to\infty$ of the process $\left(\frac{1}{N}\sum_{i=1}^N \gd_{Y_{i,t}},m_{N,t}\right)$, where
\begin{equation}
\label{eq:mN_Y}
m_{N,t}=\frac{1}{N}\sum_{i=1}^N X_{i,t}, \quad\text{and}\quad Y_{i,t} = X_{i,t}-m_{N,t}.
\end{equation}
In this set-up, $p_t$ is here the fast variable, while $m_t$ is the slow one. For $\gd=0$, this system reduces to
\begin{equation}
\left\{
\begin{array}{rcl}
\partial_t p^0_t &=& \cL p^0_t \\
 m^0_t &= & m_0 
\end{array}
\right. , 
\end{equation}
so $p^0_t=e^{t\cL}p_0$ is the distribution of an Ornstein-Uhlenbeck process, and thus converges exponentially fast to $\rho$, the density of the Gaussian distribution on $\mathbb{R}^d$ with mean $0$ and variance $\gs^2K^{-1}$ (see Proposition~\ref{prop:control_semigroup theta=1} for more details on the contraction properties of $\cL$). So heuristically, taking $\gd$ small, at first approximation $p_t$ stays close to $\rho$ while $m_t$ satisfies
\begin{equation}\label{eq:approx dot mt}
\dot m_t\approx \gd \int_{\bbR^d} F_{m_t}(x)\rho(x) \dd x = \gd \int_{\bbR^d}F(x)\rho(m_t-x)\dd x=\gd (F*\rho)(m_t).
\end{equation}
For the non-centered PDE \eqref{eq:PDE non centered} this approximation means that $u_t$ is close to a Gaussian distribution with variance variance $\gs^2K^{-1}$ and mean $m_t$, where the dynamics of $m_t$ is governed at first order by \eqref{eq:approx dot mt}. Following this heuristic, we expect a periodic behavior for the system \eqref{eq:syst PDE gd} if the approximate dynamics of $m_t$ is itself periodic. In this spirit, the main hypothesis we will adopt below is that the solution $z_t$ to
\begin{equation}\label{eq:zt}
\dot z_t = \gd \int_{\bbR^d} F_{z_t}(x)\rho(x) \dd x =\gd\left\langle F_{z_t},\rho\right\rangle 
\end{equation}
admits a stable periodic solution $(\alpha^\gd_t)_{t\in[0,\frac{T_\ga}{\gd}]}$ (more details on this notion of stability, in terms of Floquet exponents, will be given in Section~\ref{sec:main results}). In Proposition~\ref{prop: satisfies hyp of Bates} we will show that with these hypotheses, the manifold $\tilde \cM^\gd=(\rho,\ga^\gd_t)_{t\in [0,T_\ga/\gd]}$ is approximately invariant for \eqref{eq:syst PDE gd}.

\medskip

Let us now describe a situation where the above heuristics is true: in \cite{Lucon:2018qy, 2018arXiv181100305L} we considered the classical FitzHugh-Nagumo model
defined by $d=2$ and
\begin{equation}
F(x,y) = \left(x-\frac{x^3}{3}-y,\frac1c\left(x+a-by \right) \right).
\end{equation}
A direct calculation shows that in that case, with $K=\text{diag}(k_1,k_2)$ and $\gs=\text{diag}(\gs_1,\gs_2)$, 
 \begin{equation}
 \label{eq:Fmod_FHN}
\int_{\bbR^d} F_{z_1,z_2}(x,y)\rho(x,y) \dd x\, \dd y  = \left(\left(1-\frac{\gs_1^2}{k_1}\right)z_1-\frac{z_1^3}{3}-z_2,\frac1c\left(z_1+a-bz_2 \right) \right),
\end{equation}
which defines again a FitzHugh-Nagumo model. The additional factor $ \frac{ \sigma_{1}^{ 2}}{ k_{ 1}}$ in \eqref{eq:Fmod_FHN} reflects the influence of noise and interaction in the mean-field system \eqref{eq:McKean}. For an accurate choice of parameters (take e.g. $a=\frac13$, $b=1$ and $c=10$), it can be shown that the dynamics of the mean value \eqref{eq:zt} has a unique steady state when $ \frac{ \sigma_{ 1}^{ 2}}{ k_{ 1}}=0$ whereas it admits a stable periodic solution for $\frac{\gs_1^2}{k_1}$ not too small and not too large, for example $\frac{\gs_1^2}{k_1}=0.2$.  We refer to \cite{Lucon:2018qy}, \S~3.4 for a more details of the corresponding bifurcations). The purpose of \cite{Lucon:2018qy,2018arXiv181100305L} was to show that the heuristics developed above is true, i.e. the periodicity of \eqref{eq:zt} propagates to \eqref{eq:syst PDE gd}. This emergence of periodic behavior induced by noise and interaction is a signature of excitability: the system \eqref{eq:PDE non centered} exhibits a periodic behavior induced by the combined effect of the noise and the interaction, which is not present in the isolated system $\dot z_t=F(z_t)$. We refer to \cite{Lucon:2018qy} for a discussion and references on this phenomenon.

\medskip

As already said, the point of the present work is to go beyond the existence of oscillations for \eqref{eq:PDE non centered}, that is to prove regularity for the dynamics around such a limit cycle. Unfortunately the FitzHugh Nagumo model does not satisfy the hypotheses of this present work, since it has polynomial growth at infinity.  
However it is easy to see that if $\psi:\bbR_+\rightarrow \bbR_+$ is a smooth non-increasing function that satisfies $\psi(t)=1$ for $t\leq 1$ and $\psi(t)=0$ for $t\geq 2$, then for any $\gep>0$ the function $x\mapsto F(x)\psi(\gep|x|)$ satisfies our hypotheses, and that $z\mapsto\int_{\bbR^d}F_z(x)\psi(\gep|x+z|)\rho(x)dx$ converges to $z\mapsto \int_{\bbR^d}F_z(x)\rho(x)dx$ in $C^1(\mathcal{B}(0,R), \bbR^d)$ for any ball $\mathcal{B}(0,R)$ centered at $0$ with radius $R$. So, relying on classical results on normally hyperbolic manifolds \cite{fenichel1971persistence,fenichel1979geometric,wiggins2013normally} (a definition of this notion will be provided in Section~\ref{sec:main results}), if \eqref{eq:zt} admits a stable limit cycle, then it will also be the case replacing $F$ with $x\mapsto F(x)\psi(\gep|x|)$ for $\gep$ small enough.

\subsection{Weighted Sobolev norms}\label{sec:sobolev}

We present in this section the Sobolev spaces that we will use in the paper. Let us denote by $\left\vert x \right\vert_{ A}= \left( x\cdot A x\right)^{ 1/2}$ the Euclidean norm twisted by some positive matrix $A$, and, for any $ \theta\in \mathbb{ R}$, let us define the weight $w_\theta$ by
\begin{equation}
w_{ \theta}(x)= \exp \left( -\frac{ \theta}{ 2} \left\vert x \right\vert_{ K \sigma^{ -2}}^{ 2}\right).
\end{equation}
We denote as $L^{ 2}_{ \theta}$ the $L^{ 2}$-space with weight $ w_{ \theta}$, that is with norm
\begin{equation}
\label{eq:norm_Lw}
\left\Vert h \right\Vert_{ L^{ 2}_{ \theta}}= \left( \int_{ \mathbb{ R}^{ d}} \left\vert h(x) \right\vert^{ 2} w_{ \theta}(x) {\rm d}x\right)^{ \frac{ 1}{ 2}}.
\end{equation}
For any $\theta>0$ we consider the Ornstein-Uhlenbeck operator
\begin{equation}
\label{eq:L_OU}
\cL_{ \theta}^{ \ast} f= \nabla\cdot (\gs^2\nabla f)- \theta Kx \cdot\nabla f.
\end{equation}
It is well know (see for example \cite{Bakry:2014}) that $ \mathcal{ L}_{ \theta}^{ \ast}$ admits the following decomposition: for all $l\in \bbN^d$,
\begin{equation}\label{eq:decomp Ltheta}
\cL_{ \theta}^{ \ast} \psi_l = -\lambda_l \psi_l, \quad \text{with} \quad \lambda_l = \theta\sum_{i=1}^d k_il_i\quad \text{and}\quad  \psi_{ l}(x):=\psi_{ l, \theta}(x) = \prod_{i=1}^dh_{l_i}\left(\sqrt{\frac{ \theta k_i}{\gs_i^2}}x_i\right),
\end{equation}
where $h_n$ is the $n^\text{th}$ renormalized Hermite polynomial:
\begin{equation}
h_{ n}(x)= \frac{ \left(-1\right)^{ n}}{ \sqrt{ n!} (2 \pi)^{ \frac{ 1}{ 4}}} e^{ \frac{ x^{ 2}}{ 2}} \frac{ {\rm d}^{ n}}{ {\rm d}x^{ n}} \left\lbrace e^{ - \frac{ x^{ 2}}{ 2}}\right\rbrace.
\end{equation}
The family $( \psi_{ l, \theta})_{l\in \bbN^d}$ is an orthonormal basis of $L^2_\theta$. For $f,g$ with decompositions $f=\sum_{l\in \bbN^d}f_l \psi_l$ and $g=\sum_{l\in \bbN^d}g_l \psi_l$, we consider the scalar products
\begin{equation}
\label{eq:prod_Hr}
\langle f,g\rangle_{H^r_\theta} = \left\langle (a_\theta-\cL_\theta^{ \ast})^r f, \bar{ g}\right\rangle_{L^2_{ \theta}} = \sum_{l\in \bbN^d} (a_\theta+\lambda_l)^r f_l \bar g_l,
\end{equation}
where $a_\theta=\theta \, \Tr K$
and denote by $H^r_{ \theta}$ the completion of the space of smooth function $u$ satisfying $\Vert u\Vert_{H^r_{ \theta}}<\infty$. The choice of the constant $a_\theta$ is made to simplify some technical proofs given in the Appendix~\ref{app:OU} (see the proof of Proposition~\ref{prop:control_semigroup_Last}). Another choice of positive constant would produce an equivalent norm. From Lemma \ref{lem:derivative in Hr+1} it is clear that $\left\Vert \partial_{x_i} f\right\Vert_{H^r_{\theta}}\leq \Vert f\Vert_{H^{r+1}_\theta}$, and that, if $n\in \mathbb{N}$, the norm $\Vert f \Vert_{H^n_\theta}$ is in fact equivalent to
\begin{equation}
\sqrt{\sum_{l\in \bbN^d, \, \sum_{i=1}^d l_i \leq n}\left\Vert \partial^{l_1}_{x_1}\ldots \partial^{l_d}_{x_d} f \right\Vert^2_{L^2_\theta}}.
\end{equation}

\medskip

We denote by $H^{-r}_\theta$ the dual of $H^r_\theta$. Relying on a ``pivot" space structure (for more details, see Appendix~\ref{app:OU}), the product $\langle u,f\rangle_{H^{-r}_\theta,H^r_{\theta}}$ can be identified with the flat $L^2$ product $\langle u,f \rangle$: $ L^2_{-\theta}$ can be seen as a subset of $H^{-r}_\theta$, and for all $f\in H^r_\theta$ and $u\in L^2_{-\theta}$ we have
\begin{equation}
\langle u,f\rangle_{H^{-r}_\theta,H^r_{\theta}} = \langle u,f \rangle.
\end{equation}
This identification allows us to view the operator $\cL_\theta$ defined by
\begin{equation}\label{eq:def Ltheta}
\cL_{\theta}u =\nabla\cdot (\gs^2\nabla f)+ \nabla \cdot\left(\theta Kx  f\right),
\end{equation}
seen as an operator in $H^{-r}_\theta$, as the adjoint of $\cL^*_\theta$, seen as an operator in $H^r_\theta$.  This is in particular the case for $\cL=\cL_1$, whose contraction properties will be crucial in the results given in this paper.

\medskip

Our aim in this paper is to give the existence of a periodic solution for \eqref{eq:syst PDE gd} viewing $p_t$ as an element of $H^{-r}_\theta$. The necessity of considering $H^{-r}_\theta$ instead of simply taking $H^{-r}_1$ goes back to the companion paper \cite{LP2021b}, in which we study the long time behavior of the empirical measure $u_{N,t}$ in the same functional spaces. Since this empirical measure involves a sum of Dirac distributions, it can be seen as an element of $H^{-r}_\theta$ for $r>d/2$, and we have $\Vert \gd_x\Vert_{H^{-r}_\theta}\leq C w_{\frac{\theta}{4-\eta}}(x)$ for $\eta>0$ (see Lemma~2.1 in \cite{LP2021b}). Some moment estimates lead us in \cite{LP2021b} to bound terms of the form $\bbE\left[w_{\frac{m\theta}{4-\eta}}(Y_{i,t})\right]$ with $m$ large and $Y_{i,t}$ defined in \eqref{eq:mN_Y}. Since we consider cases where $Y_{i,t}$ has a distribution close to $\rho$, for this expectation to be bounded we need to consider small values of $\theta$. We need therefore to work in $ H_{ \theta}^{ -r}$ for general $ \theta$ and not only for $ \theta=1$.

\bigskip

It is well known that the semi-group $e^{t\cL}$ satisfies, for $\lambda <  \min(k_1,\ldots,k_d)$ and $u\in H^{-r}_1$ with $\int u =0$, the contraction property
\begin{equation}
\left\Vert e^{t\cL}u\right\Vert_{H^{-r}_1}\leq Ct^{-\frac{\ga}{2}}e^{-t\lambda}\left\Vert u\right\Vert_{H^{-(r+\ga)}_1}.
\end{equation} 
By obtaining similar estimates (see the following Proposition, which is a particular case of the slightly more general Proposition \ref{prop:control_semigroup}), we will be able to work in the space $H^{-r}_\theta$ with any value of $\theta$ smaller than $1$, but with the constraint of considering values of $r$ larger than a $r_0>0$ (independent from $\theta$).
\begin{proposition}
\label{prop:control_semigroup theta=1}
For all $0< \theta\leq 1$ the operator $\cL$ is sectorial and generates an analytical semi-group in $H^{-r}_\theta$. Moreover we have the following estimates: for any $ \alpha\geq0$, $r\geq 0$ and $\lambda <  \min(k_1,\ldots,k_d)$ there exists a constant $C_\cL>0$ such that for all $u\in H_{ \theta}^{ -(r+\ga)}$,
\begin{equation}\label{eq:bound_semigroup}
\left\Vert e^{t\cL} u\right\Vert_{H^{-r}_{ \theta}} \leq C_\cL \left(1+t^{- \alpha/2} e^{ - \lambda t}\right)\Vert u\Vert_{H^{-(r+\ga)}_{ \theta}},
\end{equation}
and for $r\geq 1$,
\begin{equation}\label{eq:bound_semigroup_mean0}
\left\Vert e^{t\cL} \nabla u\right\Vert_{H^{-r}_{ \theta}} \leq C_\cL t^{- \frac12} e^{- \lambda t}\Vert u\Vert_{H^{-r}_{ \theta}}\, .
\end{equation}
Moreover for all $r\geq 0$, $ 0<\varepsilon\leq 1$ and $s\geq0$,
\begin{equation}
\label{eq:diff_semigroup_L}
\left\Vert \left(e^{ (t+s) \mathcal{ L}}- e^{ t \mathcal{ L}}\right)u \right\Vert_{ H_{ \theta}^{ -r}}\leq C_\cL s ^{ \varepsilon}t^{-\frac12-\gep} e^{ - \lambda t} \left\Vert u \right\Vert_{ H_{ \theta}^{ -(r+1)}}.
\end{equation}
Finally, there exist $r_0>0$ such that for any $0<\theta\leq 1$, for all $r>r_0$, $t>0$ and all $u \in H^{-r}_\theta$ satisfying $\int u=0$,
\begin{equation}
\label{eq:bound semi group mean0 2}
\left\Vert e^{t\cL}u\right\Vert_{H^{-r}_\theta}\leq C_\cL e^{-\lambda t}\left\Vert  u\right\Vert_{H^{-r}_\theta}.
\end{equation}
\end{proposition}

\subsection{Main results}\label{sec:main results}

With the notation $ \mu_t:=(p_t, m_t)$ the system \eqref{eq:syst PDE gd} becomes
\begin{equation}\label{eq:EDP G}
\left\{
\begin{array}{rl}
\partial_t p_t &= \cL p_t +\gd G_1(\mu_t)\\
\dot m_t &=  \gd G_2(\mu_t) 
\end{array}
\right. ,
\end{equation}
where
\begin{equation}
G(\mu)=G(p,m)= \left(\begin{array}{c}
G_1(p,m)\\ G_2(p,m) \end{array}\right)
=
\left(\begin{array}{c}
-\nabla\cdot \left(p\left(F_m-\int F_m p\right)\right)\\ 
\int F_m p\end{array}\right).
\end{equation}
We place ourselves on the space 
$
\mathbf{ H}^{ r}_{ \theta}:= H_{ \theta}^{ r}\times \mathbb{ R}^{ d}
$
endowed with the scalar product
\begin{equation}
\left\langle (f,m)\, ,\, (g,m^{ \prime})\right\rangle_{ \mathbf{ H}^{ r}_{ \theta}}:= \left\langle f\, ,\, g\right\rangle_{ H_{ \theta}^{ r}} + m\cdot m^{ \prime}.
\end{equation}
We will denote $\mathbf{H}^{-r}_\theta$ the dual of $\mathbf{ H}^{ r}_{ \theta}$. Clearly $\mathbf{H}^{-r}_\theta=H^{-r}_\theta\times \bbR$ and, relying as above on a ``pivot" space structure, the product $\langle (\nu,h), (\phi,\psi)\rangle_{\mathbf{H}^{-r}_\theta,\mathbf{H}^{r}_\theta} $ can be identified with the flat scalar product
\begin{equation}
\llangle[\big] (\nu, h);(\varphi, \psi)\rrangle[\big]= \langle \nu, \varphi\rangle + h\cdot \psi.
\end{equation}

The following theorem states the existence and uniqueness of mild solutions of \eqref{eq:EDP G}. Its proof, given in Section~\ref{sec:proof existence}, relies on classical arguments, due to the fact that $G: \mathbf{ H}^{-r+1}_\theta\rightarrow \mathbf{ H}^{-r}_\theta$ is locally Lispchitz and $\cL$ is sectorial (see \cite{sell2013dynamics}).

\begin{theorem}\label{th:existence and regularity}
For any initial condition $ \mu=(p,m)\in \mathbf{ H}^{-r}_\theta$ with $\int_{\bbR^d} p=1$ there exists a unique maximal mild solution $ \mu_{ t}:=(p_t,m_t) = T^t( \mu)$ to \eqref{eq:EDP G} on $[0,t_c]$ for some $t_c>0$, which satisfies
$
t \mapsto T^{ t}( \mu)\in \mathcal{ C} \left(\left[0,t_c\right); \mathbf{H}^{-r}_{ \theta}\right)$.

Moreover, $ \mu\mapsto T^t(\mu)$ is $C^2$, and for any $R>0$, there exists a $\gd(R)>0$ such that for all $0\leq \gd\leq \gd(R)$ and $\mu_0=(p_0,m_0)$ satisfying $\Vert p_0-\rho\Vert_{H^{-r}_\theta}\leq R$ the solution $T^t(\mu_0)$ is well defined for all $t\geq 0$ and there exists a $C(R)>0$ such that
\begin{equation}
\sup_{t\geq 0} \Vert p_t\Vert_{H^{-r}_\theta}\leq C(R).
\end{equation} 
\end{theorem}

\begin{rem}\label{rem:mass 1}
Since we are interested in the existence of a periodic solution made of probability distributions, we will only consider initial condition $(p_0,m_0)$ satisfying $\int p_0=1$, and the conservation of mass will induce that $\int_{\bbR^d} p_t=1$ for all $t$. In the same spirit, we will only apply the differential of the semi-group $DT^t(\mu)$ to elements $\nu=(\eta,n)\in \mathbf{H}^{-r}_\theta$ that satisfy $\int_{\bbR^d} \eta=0$.
\end{rem}

As it was previously mentioned, we suppose in the following that the ordinary differential equation \eqref{eq:zt}
admits a stable periodic solution $(\alpha^\gd_t)_{t\in[0,\frac{T_\ga}{\gd}]}$. To state more precisely this hypothesis we rely on Floquet formalism (see for example \cite{teschl2012ordinary}): let us denote by $\pi^\gd_{u+t,u}$ the principal matrix solution associated to the periodic solution $\ga$, that it the solution to
\begin{equation}
\partial_t \pi^\gd_{u+t,u} = \gd \langle DF_{\ga^\gd_{u+ t}}, \rho\rangle \pi^\gd_{u+t,u},\ \pi^\gd_{ u, u}=I.
\end{equation}
The process $\pi^\gd_{u+t,u}$ characterizes the linearized dynamics around  $(\alpha^\gd_t)_{t\in[0,\frac{T}{\gd}]}$. We will suppose that this linearized dynamics is a contraction on a supplementary space of the tangent space to $(\alpha^\gd_t)_{t\in[0,\frac{T}{\gd}]}$. More precisely, we suppose that there exist projections $P^{\gd,c}_u$ and $P^{\gd,s}_u$ for all $u\in[0,\frac{T}{\gd}]$ with $u\mapsto P^{\gd,c}_u$ and $u\mapsto P^{\gd,s}_u$ smooth, that satisfy $P^{\gd,s}_u+P^{\gd,c}_u=I$ ($P^{\gd,c}_u$ being a projection on $\text{vect}(\dot \ga^\gd_u)$) and that commute with $\pi^\gd$, i.e.
\begin{equation}
P^{\gd,s}_u+P^{\gd,c}_u=I,\quad P^{\gd,s}_{u+t}\pi^\gd_{u+t,u}=\pi^\gd_{u+t,u}P^{\gd,s}_u.
\end{equation}
We emphasize here the fact that the factor $\gd$ in \eqref{eq:zt} is only responsible for a time-scale of the dynamics (the smaller $ \delta$, the slower the dynamics, the period being then $T_\ga/\gd$): we have $\ga^\gd_t=\ga^1_{\gd t}$. Its effect on the projections is only a change of parametrization: $P^{\gd,s}_u$ and $P^{\gd,c}_u$ are defined on $[0,T/\gd)$, and $P^{\gd,s}_{u/\gd}=P^{1,s}_{u}$, $P^{\gd,c}_{u/\gd}=P^{1,c}_u$ for $u\in [0,T_\ga)$.

The stability of the periodic solution $(\alpha^\gd_t)_{t\in[0,\frac{T_\ga}{\gd}]}$ is expressed by the following hypothesis: there exist positive constants $c_\ga, C_\ga$ and $\lambda_\ga$ such that
\begin{equation}
 \left| \pi^\gd_{u+t,u}P^{\gd,s}_u n\right|\leq C_\ga e^{-\gd\lambda_\ga  t}|n|\quad \text{and}\quad  c_\ga|n|\leq \left|\pi^\gd_{u+t,u}P^{\gd,c}_u n\right|\leq C_\ga |n|. \label{eq:orbit_dimd}
\end{equation}
The factor $\gd$ in the rate of contraction is again due to the fact that the factor $\gd$ in \eqref{eq:zt} slows the dynamics.

\medskip

With these hypotheses $(\alpha^\gd_t)_{t\in[0,\frac{T}{\gd}]}$ is in fact a simple example of Normally Hyperbolic Invariant Manifold (NHIM). We follow here the definition given in \cite{Bates:1998} for this concept: on a Banach space $\mathbf{X}$, a smooth compact connected manifold $\mathbf{M}$ is said to be a normally hyperbolic invariant manifold for a continuous semi flow $\mathbf{T}$ (such that $u \mapsto \mathbf{T}^t(\mu)$ is $C^1$ for all $t\geq 0$) if
\begin{enumerate}
\item $\mathbf{T}(\mathbf{M})\subset \mathbf{M}$ for all $t\geq 0$,
\item For each $m\in \mathbf{M}$ there exists a decomposition $\mathbf{X}=\mathbf{X}^c_m+\mathbf{X}^u_m+\mathbf{X}^s_m$ of closed subspaces with $ \mathbf{X}^c_m$ the tangent space to $\mathbf{M}$ at $m$,
\item For each $m\in \mathbf{M}$ and $t\geq 0$, denoting $m_1=\mathbf{T}^t(m)$, we have $D\mathbf{T}^t(m)_{|  \mathbf{X}^\iota_m}: \mathbf{X}^\iota_m\rightarrow  \mathbf{X}^\iota_{m_1}$ for $\iota=c,u,s$, and $D\mathbf{T}^t(m)_{|  \mathbf{X}^u_m}$ is an isomorphism from $\mathbf{X}^u_m$ to $\mathbf{X}^u_{m_1}$.
\item There exists a $t_0\geq 0$ and a $\lambda>0$ such that, for all $t\geq t_0$,
\begin{equation}\label{eq:NHIM1}
\lambda \inf\{|D\mathbf{T}^t(m)[x^u]|:\, x^u\in \mathbf{X}^u,\, |x^u|=1\} > \max \left\{1,\left\Vert D\mathbf{T}^t(m)_{|\mathbf{X}^c_m}\right\Vert \right\},
\end{equation} 
\begin{equation}\label{eq:NHIM2}
\lambda \min\{1,\inf |D\mathbf{T}^t(m)[x^c]|:\, x^c\in \mathbf{X}^c_m,\, |x^c|=1\}\} >\left\Vert D\mathbf{T}^t(m)_{|\mathbf{X}^s_m}\right\Vert.
\end{equation}
\end{enumerate}
\eqref{eq:NHIM1} implies that the semi flow $\mathbf{T}^t$ is expansive at $m$ in the direction $\mathbf{X}^u_m$ at a rate strictly larger than on $\mathbf{M}$, while \eqref{eq:NHIM2} shows implies that it is contractive at $m$ in the direction $\mathbf{X}^s_m$ at a rate greater than on $\mathbf{M}$.

This kind of structure is known to be robust under perturbation of the semi-flow: it has been shown in \cite{fenichel1971persistence,fenichel1979geometric} for flows in $\bbR^d$, and then generalized in \cite{hirsch1977invariant} in the case of Riemanian manifolds  and in \cite{Bates:1998,sell2013dynamics} in the infinite dimensional setting. An improvement of these classical results has been obtained in \cite{Bates:2008} by Bates, Lu and Zeng, who showed that if a system admits a manifold that is approximately invariant and approximately normally hyperbolic (a precise definition of these notions will be given in Section~\ref{sec:approx NHIM}), then the system possesses an actual normally hyperbolic invariant manifold in a neighborhood. 

\medskip

We will rely on this deep result in our work. Here, the slow fast viewpoint described in Section~\ref{sec:slow fast} suggests that for $\gd$ small  the manifold $\tilde\cM^\gd= \{ (\rho,\ga_t):\, t\in[0,T_\ga)\}$ is an approximately invariant manifold which is approximately normally hyperbolic (without unstable direction). This statement will be written rigorously in Section~\ref{sec:approx NHIM}, and proved in Section~\ref{sec:proof approx NHIM}. This idea will allow us to prove for $\gd$ small enough the existence of a stable periodic solution to \eqref{eq:syst PDE gd}, as an actual normally hyperbolic invariant manifold in a neighborhood of $ \tilde{\cM}^{ \delta}$. For a stable periodic solution, conditions \eqref{eq:NHIM1} and \eqref{eq:NHIM2} reduce to the fact that $DT^t(m)$ is bounded from above and below in the direction of the tangent space to the invariant manifold defined by the periodic solution, and is contractive on a stable direction.  

\begin{theorem}\label{th:Gamma}
There exists $\gd_0>0$ such that for $r_0$ given in Proposition~\ref{prop:control_semigroup theta=1} and for all $r\geq r_0$, $\gd \in  (0,\gd_0)$ and $\theta\in (0,1]$ the system \eqref{eq:syst PDE gd} admits a periodic solution  $\left(\Gamma_{ t}^{ \delta}\right)_{ t\in [0, T_{ \delta}]}:=(q^\gd_t,\gamma^\gd_t)_{t\in [0,T_\gd]}$ in $ \mathbf{ H}^{-r}_\theta$ with period $T_\gd>0$. Moreover $q^\gd_t$ is a probability distribution for all $t\geq 0$, and $t\mapsto \partial_t\Gamma^\gd_t$ and $t\mapsto \partial^2_t \Gamma^\gd_t$ are in $C([0,T_\gd),\mathbf{H}^{-r}_\theta)$.

Denoting $\cM^\gd=\{\Gamma^\gd_t:\, t\in [0,T_\gd)\}$ and $
\Phi_{u+s,u}(\nu) = D T^s(\Gamma^\gd_u)[\nu]
$, there exists families of projection $\Pi^{\gd,c}_u$ and $\Pi^{\gd,s}_t$ that commute with $\Phi$, i.e. that satisfy
\begin{equation}\label{eq:Phi commutes}
\Pi^{\gd,\iota}_{u+t}\Phi_{u+t,u}=\Phi_{u+t,u}\Pi^{\gd,\iota}_{u},\quad \text{for } \iota = c,s.
\end{equation}
Moreover $\Pi^{\gd,c}_t$ is a projection on the tangent space to $\cM^\gd$ at $\Gamma^\gd_t$, $\Pi^{\gd,c}_t+\Pi^{\gd,s}_t=I_d$, $t\mapsto \Pi^{\gd,c}_t \in C^1([0,T_\gd),\cB(\mathbf{H}^{-r}_\theta))$, and there exist positive constants $c_{\Phi,\gd}$, $C_{\Phi,\gd}$ and $\lambda_{\gd}$ such that 
\begin{equation}
c_{\Phi,\gd}\left\Vert \Pi^{\gd,c}_{u}(\nu)\right\Vert_{\mathbf{H}^{-r}_\theta}\leq \left\Vert \Phi_{u+t,u}\Pi^{\gd,c}_{u}(\nu)\right\Vert_{\mathbf{H}^{-r}_\theta}
\leq C_{\Phi,\gd}\left\Vert \Pi^{\gd,c}_{u}(\nu)\right\Vert_{\mathbf{H}^{-r}_\theta},
\end{equation}
\begin{equation}\label{eq:Phi contracts}
\left\Vert \Phi_{u+t,u}\Pi^{\gd,s}_{u}(\nu)\right\Vert_{\mathbf{H}^{-r}_\theta}\leq C_{\Phi,\gd}\, t^{-\frac{\ga}{2}}e^{-\lambda_{\delta} t}\left\Vert \Pi^{\gd,s}_{u}(\nu)\right\Vert_{\mathbf{H}^{-(r+\ga)}_\theta},
\end{equation}
and
\begin{equation}\label{eq:Phi bounded}
\left\Vert \Phi_{u+t,u}\nu\right\Vert_{\mathbf{H}^{-r}_\theta}\leq C_{\Phi,\gd}\left(1+t^{-\frac{\ga}{2}}e^{-\lambda_{\delta} t}\right)\left\Vert \nu\right\Vert_{\mathbf{H}^{-(r+\ga)}_\theta}.
\end{equation}
\end{theorem}

\begin{rem}
The invariant manifold $\cM^\gd$ is located at a distance of order $\gd$ from $\tilde\cM^\gd$ and the period $T_\gd$ is close to $T_\ga/\gd$ (the period of the slow system \eqref{eq:zt}). Moreover $\lambda_{\Phi,\gd}$ is of order $\gd$ due to the fact that $z_t$ contracts around $\ga_t$ with rate $\gd\lambda_\ga$.
\end{rem}

In \cite{Bates:2008} it is in addition proven that the stable manifold of the actual NHIM (in our case $\cM^\gd$ is attractive, the stable manifold is in fact a neighborhood $\cW^\gd$ of $\cM^\gd$) is foliated by invariant foliations: $\cW^\gd=\cup_{m\in \cM^\gd} \cW^\gd_m $, where $\nu \in \cW^\gd_m$ if and only if $T^t(\nu)-T^t(m)$ converges to $0$ exponentially fast. This implies the existence of an isochron map $\Theta^\gd:\cW^\gd\rightarrow \bbR/T_\gd\bbZ$ that satisfies $\Theta^\gd(\nu)=t$ if $\nu\in \cW^\gd_{\Gamma^\gd_t}$. The deep general result of \cite{Bates:2008} ensures that $\Theta^\gd$ is Hölder continuous, which is not entirely satisfying in view of the companion paper \cite{LP2021b}, in which we aim to apply Itô's Lemma to $\Theta^\gd(u_{N,t})$. However, the fact that in the present case we simply deal with a stable periodic solution allow us to prove that $\Theta^\gd$ has in our particular case $C^2$ regularity, as stated in the following theorem.

\begin{theorem}\label{th:Theta}
For $r$ and $\gd$ as in Theorem~\ref{th:Gamma}, there exists neighborhood $\cW^\gd\in \mathbf{H}^{-r}_{\theta}$ of $\cM^\gd$ and a $C^2$ mapping $\Theta^\gd:\cW^\gd\rightarrow \bbR/T_\gd\bbZ$ that satisfies, for all $\mu \in \cW^\gd$, denoting $\mu_t=T^t \mu$,
\begin{equation}
\Theta^\gd(\mu_t)=\Theta^\gd(\mu)+t \quad \text{mod } T_\gd,
\end{equation}
and there exists a positive constant $C_{\Theta,\gd}$ such that, for all $\mu\in \cW^\gd$ with $\mu_t=T^t\mu$,
\begin{equation}
\left\Vert \mu_t - \Gamma^\gd_{\Theta^\gd(\mu)+t}\right\Vert_{\mathbf{H}^{-r}_\theta}\leq C_{\Theta,\gd} e^{-\lambda_\gd t}\left\Vert \mu - \Gamma^\gd_{\Theta^\gd(\mu)}\right\Vert_{\mathbf{H}^{-r}_\theta}.
\end{equation}
Moreover $\Theta^\gd$ satisfies, for all $\mu \in \cW^\gd$,
\begin{equation}
\left\Vert D^2\Theta^\gd(\mu) - D^2\Theta^\gd\left(\Gamma^\gd_{\Theta^\gd(\mu)}\right)\right\Vert_{\cB\cL (\mathbf{H}^{-r}_\theta)}\leq C_{\Theta,\gd} \left\Vert \mu - \Gamma^\gd_{\Theta^\gd(\mu)}\right\Vert_{\mathbf{H}^{-r}_\theta}.
\end{equation}
\end{theorem}
 
\subsection{An approximately invariant manifold that is approximately normally hyperbolic}\label{sec:approx NHIM}

In view of the slow fast formalism described in Section~\ref{sec:slow fast}, our aim is to view $\tilde \cM^\gd=\left\{( \rho,\alpha_t): t\in[0,\frac{T}{\gd})\right\}$ as an approximately invariant and approximately normally hyperbolic manifold, in the sense of \cite{Bates:2008}. 

In fact the result of \cite{Bates:2008} is stated for dynamical systems taking values in a Banach space, while we will consider here solutions $(p_t,m_t)$ to \eqref{eq:syst PDE gd} elements of $\mathbf{H}^{-r}_\theta$ that satisfy $\int_{\bbR^d} p_t=1$ (since we are interested in probability distributions, recall Remark~\ref{rem:mass 1}), so we will rather consider an affine space. It will not pose any problem, since then $(p_t-\rho,m_t)$ is an element of $\left\{(v,m)\in \mathbf{H}^{-r}_\theta:\, \int_{\bbR^d}v=0\right\}$ which is a Banach space.

\medskip

Following the notations of \cite{Bates:2008} we set 
\begin{equation}
\label{eq:psi}
\psi(t):=( \rho,\alpha^\gd_{t}),\ t\in \bbR/T\gd\bbZ.
\end{equation} 
With this notation we have $\tilde \cM^\gd=\psi(\bbR/T\gd\bbZ)$.
We will consider the projections $\tilde \Pi^{\gd,s}_u$ and $\tilde \Pi^{\gd,c}_u$ defined for $(p,m)\in \mathbf{H}^{-r}_\theta$ by
\begin{equation}
\tilde \Pi^{\gd,s}_u(p,m) = (p, P^{\gd,s}_u m),\qquad \tilde \Pi_u^{\gd,c}(p,m) =(0,P^{\gd,c}_u m),
\end{equation}
where $P^{\gd,s}_t$ and $P^{\gd,c}_t$ are the projections defined in Section~\ref{sec:main results}. The subspaces $\tilde{ \mathbf{X}}^{\gd,c}_u = \tilde \Pi_u^{\gd,c} (\mathbf{H}^{-r}_\theta)$ and $ \tilde{\mathbf{X}}^{\gd,s}_u = \tilde \Pi^{\gd,s}_u (\mathbf{H}^{-r}_\theta)$ will correspond to the approximately tangent space and stable space of $\tilde\cM^\gd$. It is clear that for each $t\in[0,\frac{T}{\gd})$ we have 
\begin{equation}\label{eq:decomp H-r}
\mathbf{H}^{-r}_\theta=\tilde {\mathbf{X}}_t^{\gd,c}\oplus \tilde {\mathbf{X}}_t^{\gd,s}.
\end{equation}

\medskip

Consider $\tau$ such that 
\begin{equation}\label{eq:hyp tau}
 e^{-\lambda_\ga \tau} \leq \frac{c_\ga}{8 C_\ga},
\end{equation}
where $c_\ga, C_\ga, \lambda_{ \alpha}$ are given by \eqref{eq:Phi contracts}. The following proposition states that $\tilde\cM^\gd$ satisfies the hypotheses given in \cite{Bates:2008}, making it an approximately invariant and approximately normally hyperbolic manifold.

\begin{proposition}\label{prop: satisfies hyp of Bates}
There exists $\gd_0>0$ such that for $r_0$ given in Proposition~\ref{prop:control_semigroup theta=1} and for all $r\geq r_0$, $\gd\in(0, \gd_0)$ and $\theta\in(0,1)$, the following assertions are true.
\begin{enumerate}
\item (Definition 2.1. in \cite{Bates:2008}) There exists a positive constant $\kappa_1$ such that for all $u\in \bbR/\frac{T}{\gd}\bbZ$,
\begin{equation}
\left\Vert T^{\frac{\tau}{\gd}}(\rho,\ga_u) - (\rho,\ga_{u+\frac{\tau}{\gd}})\right\Vert_{\mathbf{H}^{-r}_\theta} \leq \kappa_1 \gd.
\end{equation} 
\item  (Hypothesis (H2) in \cite{Bates:2008}) There exist positive constants $\kappa_2,\kappa_3,\kappa_4$ such that for all $s, t\in \bbR/\frac{T}{\gd}\bbZ$ such that $|s-u|\leq 1, \, |t-u|\leq 1$, and $\iota=s,c$,
\begin{equation}
\left\Vert \tilde\Pi^{\gd,\iota}_u\right\Vert_{\cB(\mathbf{H}^{-r}_\theta)}\leq \kappa_2, \qquad  \left\Vert \tilde\Pi^{\gd,\iota}_u-\tilde \Pi^{\gd,\iota}_s\right\Vert_{\cB(\mathbf{H}^{-r}_\theta)}\leq \kappa_3 \left\Vert \psi(t)-\psi(s)\right\Vert_{\mathbf{H}^{-r}_\theta},
\end{equation}
and
\begin{equation}
\frac{\left\Vert\psi(t)-\psi(s)-\tilde\Pi^{\gd,c}_s(\psi(t)-\psi(s))\right\Vert_{\mathbf{H}^{-r}_\theta}}{\left\Vert \psi(t)-\psi(s)\right\Vert_{\mathbf{H}^{-r}_\theta}}\leq \kappa_4 \gd.
\end{equation}
\item (Hypothesis H3 in \cite{Bates:2008}) There exists a positive constant $\kappa_5$ such that for all $u\in \bbR/\frac{T}{\gd}\bbZ$,
\begin{equation}
\max\left\{\left\Vert \tilde \Pi^{\gd,c}_{u+\frac{\tau}{\gd}}D T^{\frac{\tau}{\gd}}(\rho,\ga_u)_{|\tilde{\mathbf{X}}^{\gd,s}_u}\right\Vert_{\cB(\mathbf{H}^{-r}_\theta)},\left\Vert \tilde\Pi^{\gd,s}_{u+\frac{\tau}{\gd}}D T^{\frac{\tau}{\gd}}(\rho,\ga_u)_{|\tilde{\mathbf{X}}^{\gd,c}_u}\right\Vert_{\cB(\mathbf{H}^{-r}_\theta)}\right\}\leq \kappa_5\gd.
\end{equation} 
\item (Hypothesis H3' and C3 in \cite{Bates:2008}) There exist $a\in (0,1)$ and $\tilde \lambda>0$ such that for all $u\in \bbR/\frac{T}{\gd}\bbZ$,
\begin{equation}
\left\Vert \left(\tilde \Pi^{\gd,c}_{u+\frac{\tau}{\gd}}D T^{\frac{\tau}{\gd}}(\rho,\ga_u)_{|\tilde{\mathbf{X}}^{\gd,c}_u}\right)^{-1}\right\Vert_{\cB(\mathbf{H}^{-r}_\theta)}^{-1}>a,
\end{equation}
and
\begin{equation}
\left\Vert \tilde\Pi^{\gd,s}_{u+\frac{\tau}{\gd}}D T^{\frac{\tau}{\gd}}(\rho,\ga_u)_{|\tilde{\mathbf{X}}^{\gd,s}_u}\right\Vert_{\cB(\mathbf{H}^{-r}_\theta)}\leq \tilde \lambda \min\left(1, \left\Vert \left(\tilde \Pi^{\gd,c}_{u+\frac{\tau}{\gd}}D T^{\frac{\tau}{\gd}}(\rho,\ga_u)_{|\tilde{\mathbf{X}}^{\gd,c}_u}\right)^{-1}\right\Vert_{\cB(\mathbf{H}^{-r}_\theta)}^{-1}\right),
\end{equation} 
\item (Hypothesis H4 in \cite{Bates:2008})
There exist positive constants $\kappa_6$ and $\kappa_7$ such that
\begin{equation}
\left\Vert D T^{\frac{\tau}{\gd}}|_{\cV(\tilde\cM^\gd,1)}\right\Vert_{\cB(\mathbf{H}^{-r}_\theta)} \leq \kappa_6, \quad \left\Vert D^2 T^{\frac{\tau}{\gd}}|_{\cV(\tilde\cM^\gd,1)}\right\Vert_{\cB\cL\left(\left(\mathbf{H}^{-r}_\theta\right)^2,\mathbf{H}^{-r}_\theta \right)} \leq \kappa_7.
\end{equation}
where $\cV(\tilde\cM^\gd,R_0)$ denote the $R_0$-neighborhood of $\tilde\cM^\gd$.
\item (Hypothesis H5 in \cite{Bates:2008}) For any $\gep>0$ there exists $\zeta>0$ such that for all $ \mu=(p,m)\in \cV(\tilde\cM^\gd,1)$ and $t\in [\frac{\tau}{\gd},\frac{\tau}{\gd}+\zeta]$,
\begin{equation}
\left\Vert T^t( \mu) - T^{\frac{\tau}{\gd}}( \mu)\right\Vert_{\mathbf{H}^{-r}_\theta} \leq \gep. 
\end{equation}
\end{enumerate}
\end{proposition}
The first five items of Proposition~\ref{prop: satisfies hyp of Bates} focus on properties of the semi-group $\left(T^{n\frac{\tau}{\gd}}\right)_{n\geq 0}$ discretized in time, showing that $\tilde\cM^\gd$ is an approximately invariant manifold approximately normally hyperbolic for this semi-group, while the last item is an uniform in time bound that implies that this property is also true for the semi-group $\left(T^t\right)_{t\geq 0}$. More precisely $(1)$ shows that $\tilde\cM^\gd$ is approximately invariant for the discrete semi- group, $(2)$ shows that $\tilde{\mathbf{X}}^{\gd,c}_u$ is an approximation of the tangent space to $\tilde\cM^\gd$ at $(\rho,\ga_u)$ and that $\psi$ does not twist too much, $(3)$ implies that $\tilde{\mathbf{X}}^{\gd,c}$ and $\tilde{\mathbf{X}}^{\gd,s}$ are approximately invariant under $\left(D T^{n\frac{\tau}{\gd}}\right)_{n\geq 0}$, and $(4)$ implies that $\left(D T^{n\frac{\tau}{\gd}}\right)_{n\geq 0}$ contracts more in the direction $\tilde{\mathbf{X}}^{\gd,s}$ than in the direction $\tilde{\mathbf{X}}^{\gd,c}$, while it does not contract too much in the direction $\tilde{\mathbf{X}}^{\gd,c}$. $(5)$ is a technical assumption useful in their proof. 

Remark that we do not quote here the hypothesis (H1) of \cite{Bates:2008} in this Proposition, since it is simply \eqref{eq:decomp H-r}. Moreover in \cite{Bates:2008} the authors treat first the inflowing invariant case, and then the overflowing invariant case, while we are here interested in an actual invariant manifold (both inflowing and overflowing), which is why we mix hypotheses $(Hi)$ and $(C3)$, as it is done in Theorem~6.5 of \cite{Bates:2008}.

\subsection{Structure of the paper}
The proof of Theorem~\ref{th:existence and regularity} concerning the well-posedness of \eqref{eq:syst PDE gd} is carried out in Section~\ref{sec:proof existence}. Proposition~\ref{prop: satisfies hyp of Bates} is proven in Section~\ref{sec:proof approx NHIM}. The main result of existence of periodic solutions (Theorem~\ref{th:Gamma}) is proven in Section~\ref{sec:proof Gamma}. The question of regularity of the isochron is addressed in Section~\ref{sec:proof theta}. The Appendix~\ref{app:OU} gathers technical estimates on the Ornstein-Uhlenbeck operator and some Gr\"onwall type lemmas are listed in Appendix~\ref{app:Gron}.

\section{Proof of Theorem~\ref{th:existence and regularity}}\label{sec:proof existence}

We give in this section the existence, uniqueness and regularity result of Theorem~\ref{th:existence and regularity}. We rely here on classical arguments one can find for example in \cite{sell2013dynamics} or \cite{Henry:1981}.

\begin{proof}[Proof of Theorem~\ref{th:existence and regularity}]
We first remark that $G: \mathbf{ H}^{-r}_\theta\rightarrow \mathbf{ H}^{-(r+1)}_\theta$ is locally Lispchitz. Indeed, for any $(p,m)\in \mathbf{H}^{-r}_\theta$ and any $(\varphi,\psi)\in \mathbf{H}^{r+1}_\theta$,
\begin{equation}
\llangle[\big] G(p,m),(\varphi,\psi) \rrangle[\big]=-\left\langle p\left(F_m-\int F_mp\right),\nabla \varphi\right\rangle +  \psi\cdot \int F_m p.
\end{equation}
We have $\left|\int F_m p\right|\leq \Vert F_m\Vert_{H^r_\theta}\Vert p\Vert_{\mathbf{H}^{-r}_\theta}$, and due to the fact that all derivatives of $F$ are bounded, $\Vert F_m\Vert_{H^r_\theta}\leq C_F$ independently from $m$. Moreover, due to the same reason, we have $\Vert F_m\cdot \nabla \varphi\Vert_{H^r_\theta}\leq  C_F \Vert \nabla \varphi\Vert_{H^{r}_\theta}$ independently from $m$. This means that
\begin{equation}
\left|\llangle[\big] G(p,m),(\varphi,\psi) \rrangle[\big]\right|\leq C \Vert p\Vert_{\mathbf{H}^{-r}_\theta}\left( 1+\Vert p\Vert_{\mathbf{H}^{-r}_\theta}\right)\Vert  \varphi\Vert_{H^{r+1}_\theta}+C\Vert p\Vert_{\mathbf{H}^{-r}_\theta}|\psi|.
\end{equation}
We deduce $\Vert G(\nu)\Vert_{\mathbf{H^{-(r+1)}_\theta}}\leq C \Vert \nu\Vert_{\mathbf{H^{-r}_\theta}}\left(1+\Vert \nu\Vert_{\mathbf{H^{-r}_\theta}}\right)$, and thus that $G$ is locally Lipschitz.

Remark that when $p$ is a probability distribution $\left|\int F_m p\right|\leq C_F\left|\int p\right|\leq C_F$, and in this case $G$ is in fact Globally Lipschitz.

\medskip

Now, since the operator $\cL$ is sectorial in $H^{-r}_\theta$, it also the case of the operator $\tilde \cL$ in $\mathbf{H}^{-r}_\theta$ defined by $\tilde \cL (p,m)=\cL p$, and thus, applying \cite[Theorem 47.8]{sell2013dynamics}, for all initial condition $ \mu=(p,m)\in \mathbf{ H}^{-r}_\theta$ there exists a unique maximal mild solution $ \mu_{ t}:=(p_t,m_t) = T^t( \mu)$ to \eqref{eq:EDP G} defined on some time interval $[0,t_c)$ and which satisfies $t \mapsto T^{ t}( \mu)\in \mathcal{ C} \left(\left[0,t_c\right); \mathbf{H}^{-r}_{ \theta}\right)$.

\medskip

Now, for $\mu=(p,m)$ and $\nu=(\eta,n)$, the Frechet differential
\begin{multline}
DG(\mu)[ \nu]= \left(\begin{array}{c}
DG_1(\mu)[ \nu]\\ DG_2( \mu)[ \nu] \end{array}\right)\\
=
\left(\begin{array}{c}
-\nabla\cdot \left( \eta\left(F_m-\int F_m p\right)\right)
-\nabla\cdot \left(p\left(D F_m [n]-\int F_m \eta -\int D F_m[n] p\right)\right)\\ 
\int F_m \eta +\int DF_m[n]p\end{array}\right)
\end{multline}
satisfies, by similar arguments as above (in particular the fact that the derivatives of $F_m$ can be bounded independently from $m$) 
\begin{equation}\label{eq:bound DG}
\left\Vert DG(\mu)[\nu]\right\Vert_{\mathbf{H}^{-(r+1)}_\theta}\leq C \left(1+\Vert \mu\Vert_{\mathbf{H}^{-r}_\theta}\right)\Vert \nu\Vert_{\mathbf{H}^{-r}_\theta},
\end{equation}
and by \cite[Theorem 49.2]{sell2013dynamics}, $ \mu\mapsto T^t(\mu)$ is Frechet differentiable, with derivative $D T^t(\mu)[ \nu]= \nu_{ t}:=( \eta_t,n_t)$ the unique mild solution to
\begin{equation}\label{eq:EDP yn}
\left\{
\begin{array}{rl}
\partial_t \eta_t &= \cL \eta_t +\gd D G_1(\mu_{ t})[ \nu_{ t}]\\
\dot n_t &=\gd  D G_2( \mu_{ t})[ \nu_{ t}]
\end{array}
\right. .
\end{equation}
By \cite[Theorem 47.5]{sell2013dynamics} the solution $ \nu_{ t}=( \eta_t,n_t)$ to \eqref{eq:EDP yn} depends continuously on $ \mu=(p,m)$, so that the flow $T^t( \mu)$ is $C^1$. One can proceed similarly for the second derivative. We have this time, for $ \nu_{ i}=( \eta_{ i}, n_{ i})$, $i=1,2$,
\begin{align}
D^2G_1(\mu)[\nu_1,\nu_2] = &-\nabla\cdot\left(\eta_1\left(DF_m [n_2]-\int F_m \eta_2-\int DF_m[n_2]p\right)\right)\nonumber\\
&-\nabla\cdot\left(\eta_2\left(DF_m [n_1]-\int F_m \eta_1-\int DF_m[n_1]p\right)\right)\nonumber\\
&-\nabla\cdot\bigg(p\bigg(D^2F_m [n_1,n_2]-\int DF_m[n_1] \eta_2-\int DF_m[n_2]\eta_1\nonumber\\
&\qquad\qquad \quad  -\int D^2F_m[n_1,n_2]p\bigg)\bigg),
\end{align}
and
\begin{equation}
D^2G_2(\mu)[\nu_1,\nu_2] =\int DF_m[n_1]\eta_2+\int DF_m[n_2]\eta_1+\int D^2F_m[n_1,n_2]p,
\end{equation}
so that
\begin{equation}\label{eq:bound D2G}
\left\Vert D^2G(\mu)[\nu_1,\nu_2]\right\Vert_{\mathbf{H}^{-(r+1)}_\theta}\leq C \left(1+\Vert \mu\Vert_{\mathbf{H}^{-r}_\theta}\right)\Vert \nu_1\Vert_{\mathbf{H}^{-r}_\theta}\Vert \nu_2\Vert_{\mathbf{H}^{-r}_\theta},
\end{equation}
and $T^t(\mu)$ is $C^2$ with $D^2 T^t( \mu)[ \nu_1,\nu_2] = \xi_t = (\xi^1_t, \xi^2_t)$
where $ \xi_0=0$ and
\begin{equation}\label{eq:dxi}
\partial_t \xi_t = \left( \cL \xi^1_t,0\right)+\gd DG( \mu_{ t})[ \xi_t]+\gd D^2G(\mu_t)[ \nu_{1,t},\nu_{2,t}],
\end{equation}
where $ \nu_{i,t}=DT^t( \mu_{ 0})[ \nu_i]$ for $i=1,2$.

\bigskip

To prove that, for $R>0$ and $\Vert p_0-\rho\Vert_{H^{-r}_\theta}\leq R$, this solution is in fact globally defined when $\gd$ is taken small enough, remark that it satisfies 
\begin{equation}
p_t = e^{t\cL}p_0+\int_0^t e^{(t-s)\cL}\nabla\cdot (p_s(\gd F_{m_s}+\dot m_s))\dd s,
\end{equation}
and
\begin{equation}
\dot m_t= \gd \langle F_{m_t}, p_t\rangle.
\end{equation}
The estimates obtained above imply directly $|\dot m_s|\leq \gd C_F \Vert p_s\Vert_{H^{-r}_\theta}$. Using Proposition~\ref{prop:control_semigroup theta=1} we get:
\begin{align}
\label{eq:pt_moins_q}
\Vert p_t\Vert_{H^{-r}_\theta}
&\leq  C_\cL\Vert p_0\Vert_{H^{-r}_\theta}+C_1 \int_0^t \frac{e^{-\lambda(t-s)}}{\sqrt{t-s}}\left\Vert   p_s(\gd F_{m_s}+\dot m_s)\right\Vert_{H^{-r}_\theta}\dd s\\ 
&\leq  C_2\left(\Vert p_0\Vert_{H^{-r}_\theta}+\gd \int_0^t \frac{e^{-\lambda(t-s)}}{\sqrt{t-s}}\Vert p_s\Vert_{H^{-r}_\theta}\left( 1+\Vert p_s\Vert_{H^{-r}_\theta}\right)\dd s\right).
\end{align}
Denote $t_0=\inf\left\{t> 0:\, \Vert p_t\Vert_{H^{-r}_\theta}\geq 2 C_2 \left(R+ \Vert \rho\Vert_{H^{-r}_\theta}\right)\right\}$. By continuity, $t_{ 0}>0$ and for all $t\in[0, t_0]$,
\begin{equation}\label{eq:q-pt}
\Vert p_t\Vert_{H^{-r}_\theta} \leq C_2 \left(R+ \Vert \rho\Vert_{H^{-r}_\theta}\right)+  \delta\sqrt{\frac{\pi}{\lambda}} 2C_2 \left(R+ \Vert \rho\Vert_{H^{-r}_\theta}\right)\left(1+2C_2 \left(R+ \Vert \rho\Vert_{H^{-r}_\theta}\right)\right).
\end{equation}
For the choice of $ \delta>0$ sufficiently small such that $ \delta\sqrt{\frac{\pi}{\lambda}}2 \left(1+2C_2 \left(R+ \Vert \rho\Vert_{H^{-r}_\theta}\right)\right)<1$, this yields that $t_0=\infty $, so that $(p_t,m_t)$ is a global solution.
\end{proof}

\section{Proof of Proposition~\ref{prop: satisfies hyp of Bates}}\label{sec:proof approx NHIM}

In this section we give the proof of Proposition~\ref{prop: satisfies hyp of Bates} which shows that $\tilde\cM^\gd$ is an approximately invariant approximately normally hyperbolic manifold. We do not prove the assertions in the order they are given in Proposition~\ref{prop: satisfies hyp of Bates}.

\medskip

\begin{proof}[Proof of Proposition ~\ref{prop: satisfies hyp of Bates}]~\\
\textit{Proof of (1).} Take $p_0= \rho$ and $m_0=\ga_u$. We then have
\begin{equation}\label{eq:mild p-rho}
p_t- \rho = \int_0^t e^{(t-s)\cL}\nabla\cdot (p_s(\gd F_{m_s}+\dot m_s))\dd s,
\end{equation}
and
\begin{equation}
\dot m_t- \dot \ga_{u+ t} = \gd \langle F_{m_t}, p_t\rangle-\gd \langle F_{\ga_{u+ t}}, \rho\rangle.
\end{equation}
As it was already proved in the preceding section, we have $\vert \dot m_s\vert  \leq C_F \gd\Vert p_s\Vert_{H^{-r}_\theta}$, and since Theorem~\ref{th:existence and regularity} with $R=1$ implies that, choosing $\gd $ small enough, $\Vert p_t\Vert_{H^{-r}_\theta}\leq C(1) $, we get from Proposition~\ref{prop:control_semigroup theta=1},
\begin{align}
\label{eq:pt_moins_q2}
\Vert p_t- \rho\Vert_{H^{-r}_\theta}
&\leq  C_1 \int_0^t \frac{e^{-\lambda(t-s)}}{\sqrt{t-s}}\left\Vert  p_s(\gd F_{m_s}+\dot m_s)\right\Vert_{H^{-r)}_\theta}\dd s\nonumber\\ 
&\leq  C_1\gd \int_0^t \frac{e^{-\lambda(t-s)}}{\sqrt{t-s}}\Vert p_s\Vert_{H^{-r}_\theta}\left( 1+\Vert p_s\Vert_{H^{-r}_\theta}\right)\dd s\leq C_2\gd.
\end{align}
Now since
\begin{multline}
\frac{1}{\gd}(\dot m_t- \dot \ga_{u+t}) = \langle DF_{\ga_{u+t}}, \rho\rangle(m_t-\ga_{u+t}) +\langle F_{m_t}-F_{\ga_{u+ t}}-DF_{\ga_{u+ t}}(m_t-\ga_{u+ t}), \rho\rangle\\
+\langle  F_{m_t},p_t- \rho\rangle,
\end{multline}
we have the following mild representation (recall that $m_0=\ga_u$):
\begin{equation}
m_t-\ga_{u+ t} = \gd\int_0^t \pi^\gd_{u+t,u+s}\Big(\langle F_{m_s}-F_{\ga_{u+ s}}-DF_{\ga_{u+s}}(m_s-\ga_{u+s}), \rho\rangle
+ \langle  F_{m_s},p_s- \rho\rangle\Big)\dd s,
\end{equation}
which leads to (recall that the derivatives of $F$ are bounded and that \eqref{eq:pt_moins_q2} is valid for all $t\geq 0$):
\begin{equation}
| m_t -\ga_{u+t}|\leq C_3\gd \int_0^t |m_s-\ga_{u+s}|^2 \dd s + C_3\gd^2 t.
\end{equation}
Consider $t_1 = \inf\{t> 0: \, |m_t-\ga_{u+t}|\geq 2\tau C_4\gd\}$ (recall the definition of $ \tau$ in \eqref{eq:hyp tau}). By continuity, $t_{ 1}>0$ and for all $t\leq t_1$ we have
\begin{equation}
|m_t-\ga_{u+t}|\leq (4\tau^2C_3^3 \gd^3+C_3\gd^2)t, 
\end{equation}
which means that $t_1\geq \frac{\tau}{\gd}$ for $\gd$ small enough, and implies (1).

\medskip
\noindent
\textit{Proof of (2).} The two first points follow directly form the fact that the projections $P^c_u$ are smooth. For the third point we have
\begin{equation}
\frac{\left\Vert\psi(t)-\psi(s)-\tilde\Pi^{\gd,c}_s(\psi(t)-\psi(s))\right\Vert_{\mathbf{H}^{-r}_\theta}}{\left\Vert \psi(t)-\psi(s)\right\Vert_{\mathbf{H}^{-r}_\theta}}=\frac{\left|\ga^\gd_{t}-\ga^\gd_{s}-P^{\gd,c}_s(\ga^\gd_{t}-\ga^\gd_{s})\right|}{\left|\ga^\gd_{t}-\ga^\gd_{s}\right|},
\end{equation}
and since
\begin{equation}
\ga^\gd_t-\ga^\gd_s=\ga^{1}_{\gd t}-\ga^{1}_{\gd s}=\gd(t-s) \frac{d}{du}\ga^{1}_{u|u=\gd s}+O(\gd^2(t-s)),
\end{equation}
and
\begin{equation}
P^{\gd,c}_s  \frac{d}{du}\ga^{1}_{u|u=\gd s}=P^{0,c}_{\gd s} \frac{d}{du}\ga^{1}_{u|u=\gd s}= \frac{d}{du}\ga^{1}_{u|u=\gd s},
\end{equation}
the term $\frac{\left\Vert\psi(t)-\psi(s)-\tilde\Pi^{\gd,c}_s(\psi(t)-\psi(s))\right\Vert_{\mathbf{H}^{-r}_\theta}}{\left\Vert \psi(t)-\psi(s)\right\Vert_{\mathbf{H}^{-r}_\theta}}$ is indeed of order $\gd$.

\medskip

\noindent
\textit{Proof of (5).}
We choose in the following $R_0=1$. For any $(p,m)\in \cV(\tilde\cM^\gd,R_0)$, wich means in particular $\Vert p-\rho\Vert_{H^{-r}_\theta}\leq R_0$, we deduce from Theorem~\ref{th:existence and regularity}, if $\gd$ is small enough, that 
\begin{equation}
\sup_{t\geq 0} \Vert p_t\Vert_{H^{-r}_\theta}\leq C(R_0).
\end{equation}
This means in particular, since $m_t = m_0+ \gd \int_0^t \langle F_{m_s}, p_s\rangle ds$, that for $C_{ 4}, C_{ 5}>0$
\begin{equation}
\sup_{t\geq 0} |\dot m_t|\leq \gd C_4,\quad \text{and}\quad \sup_{t\in [0,\frac{\tau}{\gd}]} |m_t|\leq C_5 ,
\end{equation}
where $C_{5}$ depends on $\tau$. Now, using \eqref{eq:EDP yn} we have, with $\mu_s=(p_s,m_s)$,
\begin{equation}
\eta_t = e^{t\cL}\eta_0+\gd \int_0^t e^{(t-s)\cL}D G_1(\mu_s)[\eta_s,n_s] \dd s,
\end{equation}
and
\begin{equation}
n_t = n_0+\gd \int_0^t DG_2(\mu_s)[\eta_s,n_s]\dd s.
\end{equation}
From \eqref{eq:bound DG} and Proposition~\ref{prop:control_semigroup theta=1} (recall that $\int_{\bbR^d}\eta_0=0$, see Remark~\ref{rem:mass 1}), we obtain
\begin{equation}\label{eq:mild nut}
\Vert \eta_t\Vert_{H^{-r}_\theta}\leq C_{\cL} e^{-\lambda t}\Vert \eta_0\Vert_{H^{-r}_\theta}+C_{6}\gd \int_0^t \frac{e^{-\lambda(t-s)}}{\sqrt{t-s}}\left( \Vert \eta_s\Vert_{H^{-r}_\theta}+|n_s|\right)\dd s,
\end{equation}
and
\begin{equation}
|n_t|\leq |n_0|+C_{6}\gd \int_0\left( \Vert \eta_s\Vert_{H^{-r}_\theta}+|n_s|\right)\dd s.
\end{equation}
We deduce that, for $\nu_t=DT^t(p,m)[\nu_0]=(\eta_t,n_t)$,
\begin{equation}
\Vert \nu_t\Vert_{\mathbf{H}^{-r}_\theta}\leq C_{7}\Vert \nu_0\Vert_{\mathbf{H}^{-r}_\theta}+C_{8}\gd \int_0^t \left(1+\frac{1}{\sqrt{t-s}}\right)\Vert \nu_s\Vert_{\mathbf{H}^{-r}_\theta}\dd s.
\end{equation}
Applying Lemma~\ref{lem:GH}, we get the desired bound for the $DT^{\frac{\tau}{\gd}}$ with $\kappa_6=2C_{7}e^{3C_{8}\tau}$, when $\gd$ is small enough.

\medskip

For the second derivative, recall that $D^2 T^t( \mu)[ \nu_1,\nu_2] = \xi_t = (\xi^1_t, \xi^2_t)$, where $ \xi_0=0$ and (recall \eqref{eq:dxi})
\begin{equation}
\xi^1_t = \gd \int_0^t e^{(t-s)\cL}\left(DG_1( \mu_{ S})[ \xi_s]+D^2G_1(\mu_s)[ \nu_{1,s},\nu_{2,s}]\right)\dd s,
\end{equation}
and
\begin{equation}
\xi^2_t= \gd\int_0^t \left(DG_2( \mu_{ S})[ \xi_s]+D^2G_2(\mu_s)[ \nu_{1,s},\nu_{2,s}]\right) \dd s
\end{equation}
where $\mu_t=(p_t,m_t)$, and $ \nu_{i,t}=DT^t( \mu_{ 0})[ \nu_i]$ for $i=1,2$. This induces for $t\in [0,\frac{\tau}{\gd}]$, recalling \eqref{eq:bound DG}, \eqref{eq:bound D2G} and since  $\Vert \nu_{i,t}\Vert_{\mathbf{H}^{-r}_\theta}\leq \kappa_6\Vert \nu_{i,0}\Vert_{\mathbf{H}^{-r}_\theta}$,
\begin{equation}
\left\Vert \xi^1_t\right\Vert_{H^{-r}_\theta}\leq \gd C_{9} \int_0^t \frac{e^{-\lambda (t-s)}}{\sqrt{t-s}}\left(\left\Vert \xi_s\right\Vert_{\mathbf{H}^{-r}_\theta}+\Vert \nu_{1,0}\Vert_{\mathbf{H}^{-r}_\theta}\Vert \nu_{2,0}\Vert_{\mathbf{H}^{-r}_\theta} \right)\dd s,
\end{equation}
and
\begin{equation}
\left\vert \xi^2_t\right\vert\leq \gd C_{9} \int_0^t \left(\left\Vert \xi_s\right\Vert_{\mathbf{H}^{-r}_\theta}+\Vert \nu_{1,0}\Vert_{\mathbf{H}^{-r}_\theta}\Vert \nu_{2,0}\Vert_{\mathbf{H}^{-r}_\theta} \right)\dd s.
\end{equation}
So for $t\leq \frac{\tau}{\gd}$,
\begin{equation}
\left\Vert \xi_t\right\Vert_{\mathbf{H}^{-r}_\theta}\leq C_{10}\Vert \nu_{1,0}\Vert_{\mathbf{H}^{-r}_\theta}\Vert \nu_{2,0}\Vert_{\mathbf{H}^{-r}_\theta} +\gd C_{10}\int_0^t \left(1+\frac{e^{-\lambda (t-s)}}{\sqrt{t-s}} \right)\left\Vert \xi_s\right\Vert_{\mathbf{H}^{-r}_\theta}\dd s,
\end{equation}
and one deduces from Lemma~\ref{lem:GH} that $\left\Vert \xi_t\right\Vert_{\mathbf{H}^{-r}_\theta}\leq \kappa_7\Vert \nu_{1,0}\Vert_{\mathbf{H}^{-r}_\theta}\Vert \nu_{2,0}\Vert_{\mathbf{H}^{-r}_\theta} $ with $\kappa_7=2C_{10}e^{3C_{10}\tau}$ for $t\leq \frac{\tau}{\gd}$ and $\gd$ small enough, which concludes the proof of (5).

\medskip
\noindent
\textit{Proof of (3).}
We are now interested in $DT^{\frac{\tau}{\gd}}( \rho,\ga_u)( \eta_0,n_0)=( \eta_{\frac{\tau}{\gd}},n_{\frac{\tau}{\gd}})=\nu_{\frac{\tau}{\gd}}$. From the proof of point (3) we already know that $\sup_{t\in [0,\frac{\tau}{\gd}]}\Vert \nu_t\Vert_{\mathbf{H }^{-r}_\theta}\leq \kappa_6 \Vert \nu_0\Vert_{\mathbf{H }^{-r}_\theta}$, which means, recalling \eqref{eq:mild nut}, that
\begin{align}
\Vert \eta_t\Vert_{H^{-r}_\theta}&\leq C_\cL e^{-\lambda t}\Vert \eta_0\Vert_{H^{-r}_\theta}+C_{11} \gd \int_0^t \frac{e^{-\lambda(t-s)}}{\sqrt{t-s}}\left( \Vert \eta_0\Vert_{H^{-r}_\theta}+|n_0|\right)\dd s\nonumber\\
&\leq  C_\cL e^{-\lambda t}\Vert \eta_0\Vert_{H^{-r}_\theta}+C_{12}\gd \left( \Vert \eta_0\Vert_{H^{-r}_\theta}+|n_0|\right).\label{eq:nut}
\end{align}
Moreover, since
\begin{multline}
\frac{1}{\gd}\dot n_t =\langle DF_{\ga_{u+t}} \left[n_{ t}\right], \rho\rangle -\langle DF_{\ga_{u+t}}\left[n_{ t}\right]-DF_{m_t}\left[n_{ t}\right], \rho\rangle +\langle DF_{m_t}\left[n_{ t}\right],p_t- \rho\rangle \\
+\langle  F_{m_t}, \eta_t\rangle,
\end{multline}
we have the mild representation
\begin{multline}
 n_t =\pi^\gd_{u+t,u}n_0+\gd\int_0^t \pi^\gd_{u+t,u+s}\Big( -\langle DF_{\ga_{u+s}}\left[n_{ s}\right]-DF_{m_s}\left[n_{ s}\right], \rho\rangle \\+\langle DF_{m_s}\left[n_{ s}\right],p_s- \rho\rangle
+\langle  F_{m_s}, \eta_s\rangle \Big)\dd s.
\end{multline}
From the proof of point (1), for $t\leq \frac{\tau}{\gd}$, $\Vert p_t- \rho\Vert_{H^{-r}_\theta}$ and $|m_t-\ga_{u+t}|$ are of order $\gd$, and thus we obtain (recall also that $\sup_{t\in [0,\frac{\tau}{\gd}]}|n_t|\leq \kappa_6 \Vert \nu_0\Vert_{\mathbf{H }^{-r}_\theta}$):
\begin{align}
\left|n_t-\pi^\gd_{u+t,u}n_0\right|&\leq C_{13}\gd \int_0^t\left(  \Vert \eta_s\Vert_{H^{-r}_\theta}+\gd |n_0|\right)\dd s\nonumber\\
&\leq  C_{13}\gd  \int_0^t\left( C_\cL e^{-\lambda s}\Vert \eta_0\Vert_{H^{-r}_\theta}+C_{12}\gd \left( \Vert \eta_0\Vert_{H^{-r}_\theta}+|n_0|\right)+\gd |n_0|\right)\dd s\nonumber\\
&\leq C_{14} \gd\left( \Vert \eta_0\Vert_{H^{-r}_\theta}+|n_0|\right). \label{eq:bound nt- pi n0}
\end{align}

Suppose now that $(\eta_0,n_0)\in \tilde{\mathbf{X}}^{\gd,s}_u$, that is $P^{\gd,c}_u n_0=0$. Then we have $P^{\gd,c}_{u+\frac{\tau}{\gd}}\pi^\gd_{u+\frac{\tau}{\gd},u}n_0=P^{\gd,c}_u n_0=0$,
and thus, recalling \eqref{eq:bound nt- pi n0} and \eqref{eq:nut},
\begin{equation}
\left| P^{\gd,c}_{u+\frac{\tau}{\gd}} n_{\frac{\tau}{\gd}}\right|=\left| P^{\gd,c}_{u+\frac{\tau}{\gd}}\left( n_{\frac{\tau}{\gd}}-\pi^\gd_{u+\frac{\tau}{\gd},u}n_0\right)\right| \leq C_{15}\gd \left(  \Vert \eta_0\Vert_{H^{-r}_\theta}+|n_0|\right). 
\end{equation}
This shows that
\begin{equation}
\left\Vert \tilde \Pi^{\gd,c}_{u+\frac{\tau}{\gd}}D T^{\frac{\tau}{\gd}}( \rho,\ga_u)|_{\tilde{\mathbf{X}}^{\gd,s}_u}\right\Vert_{\cB(\mathbf{H}^{-r}_\theta)}\leq C_{15}\gd .
\end{equation}

On the other hand, suppose that $( \eta_0,n_0)\in \tilde{\mathbf{X}}^{\gd,c}_u$, that is $ \eta_0=0$ and $P^{\gd,s}_u n_0=0$. We then have directly $\left\Vert \eta_{\frac{\tau}{\gd}}\right\Vert_{H^{-r}_\theta}\leq C_{12}\gd|n_0|$, and since $P^{\gd,s}_{u+\frac{\tau}{\gd}}\pi^\gd_{u+\frac{\tau}{\gd},u}n_0=P^{\gd,s}_u n_0=0$, from \eqref{eq:bound nt- pi n0} we deduce
\begin{equation}
\left| P^s_{u+\frac{\tau}{\gd}} n_{\frac{\tau}{\gd}}\right|=\left| P^s_{u+\frac{\tau}{\gd}}\left( n_{\frac{\tau}{\gd}}-\pi^\gd_{u+\frac{\tau}{\gd},u}n_0\right)\right|\leq C_{16}\gd^2\int_0^{\frac{\tau}{\gd}}|n_0|\dd s \leq C_{16}\tau\gd |n_0|. 
\end{equation}
This means that
\begin{equation}
\left\Vert \tilde \Pi^{\gd,s}_{u+\frac{\tau}{\gd}}D T^{\frac{\tau}{\gd}}( \rho,\ga_u)|_{\tilde{\mathbf{X}}^{\gd,c}_u}\right\Vert_{\cB(\mathbf{H}^{-r}_\theta)}\leq (C_{12}+C_{16}\gd)\gd. 
\end{equation}

\medskip
\noindent
\textit{Proof of (4).}
On one hand consider $(\eta_0,n_0)\in \tilde{\mathbf{X}}^{\gd,s}_u$, that is $P^{\gd,c}_u n_0=0$. Then, considering $\gd$ small enough such that $C_\cL e^{-\lambda \frac{\tau}{\gd}}\leq C_{12}\gd$, by \eqref{eq:nut} we obtain
\begin{equation}
\left\Vert \eta_{\frac{\tau}{\gd}}\right\Vert_{H^{-r}_\theta} \leq 2C_{12}\gd \left(\Vert \eta_0\Vert_{H^{-r}_\theta}+|n_0|\right).
\end{equation}
Moreover, since $P^{\gd,s}_{u+\frac{\tau}{\gd}}\pi^\gd_{u+\frac{\tau}{\gd},u}n_0=\pi^\gd_{u+\frac{\tau}{\gd},u}n_0$ and $P^{\gd,c}_un_0=0$ we obtain, by \eqref{eq:bound nt- pi n0} and \eqref{eq:orbit_dimd},
\begin{align}
\left|P^s_{u+\frac{\tau}{\gd}}n_{\frac{\tau}{\gd}}\right| & \leq \left|P^s_{u+\frac{\tau}{\gd}}\pi^\gd_{u+\frac{\tau}{\gd},u}n_0\right|  + \left|P^s_{u+\frac{\tau}{\gd}}\left(n_{\frac{\tau}{\gd}}-\pi^\gd_{u+\frac{\tau}{\gd},u}n_0\right)\right| \\
&\leq C_\ga e^{-\lambda_\ga \tau }|n_0|+C_{17}\gd \left(\Vert \eta_0\Vert_{H^{-r}_\theta}+|n_0|\right).
\end{align}
We deduce that for $\gd$ small enough
\begin{equation}
\left\Vert \tilde \Pi^{\gd,s}_{u+\frac{\tau}{\gd}}D T^{\frac{\tau}{\gd}}( \rho,\ga_u)|_{\tilde{\mathbf{X}}^{\gd,s}_u}\right\Vert_{\cB(\mathbf{H}^{-r}_\theta)}\leq 2C_\ga e^{-\lambda_\ga \tau}.
\end{equation}

On the other hand consider $(\eta_0,n_0)\in \tilde{\mathbf{X}}^{\gd,c}_u$, which means $ \eta_0=0$ and $P^{\gd,s}_un_0=0$. Then similar arguments as above (recall that this time $ \eta_0=0$) lead to
\begin{equation}
\left| P^{\gd,c}_{u+\frac{\tau}{\gd}}\left( n_{\frac{\tau}{\gd}}-\pi^\gd_{u+\frac{\tau}{\gd},u}n_0\right)\right|\leq C_{18}\gd|n_0|.
\end{equation}
We then obtain, for $\gd$ small enough, recalling \eqref{eq:orbit_dimd},
\begin{equation}
\left| P^c_{u+\frac{\tau}{\gd}} n_{\frac{\tau}{\gd}}\right|\geq \left(c_\ga-C_{18} \gd\right)|n_0|\geq \frac{c_\ga}{2} |n_0|.
\end{equation}
This means in particular that $\tilde \Pi^{\gd,c}_{u+\frac{\tau}{\gd}}D T^{\frac{\tau}{\gd}}( \rho,\ga_u)|_{\tilde{\mathbf{X}}^{\gd,c}_u}$, which is a linear mapping in finite dimensional spaces, is invertible and satisfies
\begin{equation}
\left\Vert \left(\tilde \Pi^{\gd,c}_{u+\frac{\tau}{\gd}}D T^{\frac{\tau}{\gd}}( \rho,\ga_u)|_{\tilde{\mathbf{X}}^{\gd,c}_u}\right)^{-1}\right\Vert_{\cB(\mathbf{H}^{-r}_\theta)} \leq \frac{2}{c_\ga}.
\end{equation}
We deduce (4) with $a=\frac{c_a}{4}$ and $\tilde \lambda=\frac{4C_\ga e^{-\lambda_\ga \tau}}{c_\ga}$, recalling \eqref{eq:hyp tau}.

\medskip

\noindent
\textit{Proof of (6).}
For any initial condition $ \mu=(p_{ 0}, m_{ 0})\in \cV(\tilde\cM^\gd, 1)$ recall that Theorem~\ref{th:existence and regularity} implies $ \sup_{ t\geq0} \left\Vert p_{ t} \right\Vert_{ H_{ \theta}^{ -r}} \leq C(1)$. 
Then for $ \frac{ \tau}{ \delta} \leq t <t^{ \prime}$, $ t^{ \prime}- t \leq \zeta$, for some $ \zeta\leq 1$ to be chosen later, relying on \eqref{eq:mild p-rho}, the following is true:
\begin{align}
\left\Vert p_{ t^{ \prime}} -p_{ t}\right\Vert_{ H_{ \theta}^{ -r}} \leq & \left\Vert \left(e^{ t^{ \prime} \mathcal{ L}}-e^{ t \mathcal{ L}}\right)p_{ 0} \right\Vert_{ H_{ \theta}^{ -r}}\nonumber\\
&+ \int_{ 0}^{t} \left\Vert \left(e^{ (t^{ \prime}-s) \mathcal{ L}} - e^{ (t-s) \mathcal{ L}}\right) \nabla\cdot (p_s(\gd F_{m_s}+\dot m_s))\right\Vert_{ H_{ \theta}^{ -r}} \dd s\nonumber\\
&+ \int_{ t}^{t^{ \prime}} \left\Vert e^{ (t^{ \prime}-s) \mathcal{ L}} \nabla\cdot (p_s(\gd F_{m_s}+\dot m_s))\right\Vert_{ H_{ \theta}^{ -r}} \dd s \label{aux:item7_0}
\end{align}
Using Proposition~\ref{prop:control_semigroup theta=1}, the first term above may be bounded as
\begin{align}
\left\Vert \left(e^{ t^{ \prime} \mathcal{ L}}- e^{ t \mathcal{ L}}\right)p_{ 0} \right\Vert_{ H_{ \theta}^{ -r}} &\leq C_\cL (t^{ \prime}- t)^{ \varepsilon} \frac{ e^{ -\lambda t}}{ t^{ \frac{ 1}{ 2} + \varepsilon}} \left\Vert p_{ 0} \right\Vert_{ H^{ -(r+1)}_\theta} \leq C_{19}\zeta^{ \varepsilon} \delta^{ \frac{ 1}{ 2}+ \varepsilon}\frac{ e^{ -\lambda \frac{ \tau}{ \delta}}}{ \tau^{ \frac{ 1}{ 2} + \varepsilon}}.\label{aux:item7_1}
\end{align}
Concerning the second term,
\begin{align}
\int_{ 0}^{t} \Big\Vert \left(e^{ (t^{ \prime}-s) \mathcal{ L}} - e^{ (t-s) \mathcal{ L}}\right) &\nabla\cdot (p_s(\gd F_{m_s}+\dot m_s))\Big\Vert_{ H_{ \theta}^{ -r}} {\rm d}s \nonumber \\
& \leq C_\cL (t^{ \prime}- t)^{ \varepsilon}\int_{ 0}^{t} \frac{ e^{ - \lambda (t-s)}}{ (t-s)^{ \frac{ 1}{ 2}+ \varepsilon}}\left\Vert p_s(\gd F_{m_s}+\dot m_s)\right\Vert_{ H_{ \theta}^{ -r}} \dd s \nonumber\\ 
&\leq C_{20} \delta (t^{ \prime}- t)^{ \varepsilon} \int_{ 0}^{t} \frac{ e^{ - \lambda (t-s)}}{ (t-s)^{ \frac{ 1}{ 2}+ \varepsilon}}\left\Vert p_s\right\Vert_{ H_{ \theta}^{ -r}} \left(1+\left\Vert p_s \right\Vert_{ H_{ \theta}^{ -r}}\right) \dd s\nonumber\\
&\leq C_{21} \gd \zeta^\gep. \label{aux:item7_2}
\end{align}
Now turning to the third term, relying again on Proposition~\ref{prop:control_semigroup theta=1},
\begin{align}
\int_{ t}^{t^{ \prime}} \Big\Vert e^{ (t^{ \prime}-s) \mathcal{ L}} \nabla\cdot (p_s(\gd F_{m_s}&+\dot m_s))\Big\Vert_{ H_{ \theta}^{ -r}} {\rm d}s\nonumber\\
&\leq C_{22} \delta\int_{ t}^{t^{ \prime}} \frac{ e^{ - \lambda(t^{ \prime}-s)}}{\sqrt{t'-s}}\left\Vert p_s\right\Vert_{ H_{ \theta}^{ -r}} \left(1+\left\Vert p_s\right\Vert_{ H_{ \theta}^{ -r}}\right) {\rm d}s, \nonumber\\
&\leq C_{23}\gd \zeta^{ \frac{ 1}{ 2}}. \label{aux:item7_3}
\end{align}
Gathering \eqref{aux:item7_1}, \eqref{aux:item7_2}, \eqref{aux:item7_3} into \eqref{aux:item7_0} yields
\begin{equation}
\left\Vert p_{ t^{ \prime}} -p_{ t}\right\Vert_{ H_{ \theta}^{ -r}} \leq \frac{\xi}{2}
\end{equation}
if $ \zeta\leq 1$ is chosen sufficiently small.
\medskip
We now turn turn to the control of the mean: since $\dot m_t =  \gd \int F_{m_t} d p_t $ we have that for $t\leq t^{ \prime} \leq t + \zeta$,
\begin{align*}
m_{ t^{ \prime}}-m_{ t} &= \delta \int_{ t}^{t^{ \prime}} \langle F_{ m_{ s}} ,p_{ s}\rangle \dd s.
\end{align*}
Since we have the uniform bound $ \sup_{ s\geq 0}\left\Vert p_{ s} \right\Vert_{ H_{ \theta}^{ -r}}\leq C(1)$ and since $F$ and its derivatives are bounded, the above quantity is easily bounded by some $C \delta (t^{ \prime}-t)$ which can be made smaller than $ \xi/2$, provided $ \zeta$ is taken small enough.
\end{proof}

\section{Proof of Theorem~\ref{th:Gamma}}\label{sec:proof Gamma}

\begin{proof}[Proof of Theorem~\ref{th:Gamma}]
From Proposition~\ref{prop: satisfies hyp of Bates} we know that the hypotheses needed in \cite{Bates:2008} are satisfied for $\gd$ small enough, which means that the system \eqref{eq:syst PDE gd} admits a stable normally hyperbolic manifold $\cM^\gd$ that is at distance $\gd$ from $\tilde\cM^\gd$. Indeed in \cite{Bates:2008} some constants $\eta, \chi, \gs$ need to be small for their result to be true, but in our case these constants are of order $\gd$, so we only need to suppose $\gd$ small enough. Moreover $\cM^\gd$ is contructed at a distance $\gd_0$ from $\tilde\cM^\gd$, with (see \cite{Bates:2008}, Theorem 4.2) $\gd_0$ chosen such that $\eta/\gep$ and $\gep/\gd_0$ are bounded for some $\gep>0$. Since $\eta$ is in our case of order $\gd$, we can take $\gd_0$ of order $\gd$, and $\cM^\gd$ is indeed at distance $\gd$ from $\tilde\cM^\gd$.

The invariant manifold $\cM^\gd$ is one dimensional, since $\tilde\cM^\gd$ is, so to prove that it corresponds to a periodic solution it is sufficient to prove that it does not possess any invariant point. But for any $(p_0,m_0)\in \cM ^\gd$ we have, since $\Vert p-\rho\Vert_{H^{-r}_\theta}$ and $|m_0-\ga^\gd_u|$ are of order $\gd$ for some $u\in [0,\frac{T_\ga}{\gd}]$,
\begin{equation}
\dot m_0=\gd \int F_{m_0} p_0 = \gd \int F_{\ga_u} \rho+O(\gd^2) = \dot \ga_u + O(\gd^2).
\end{equation}
Since there exists $c>0$ such that $|\dot \ga^\gd_u/\gd|>c$ independently from $u$, we have $\dot m_0\neq 0$ for the solutions starting from any point of $\cM^\gd$, which means that $\cM^\gd$ does not possess any fixed-point, and is thus defined by a periodic solution of positive period $T_\gd$, that we denote $\Gamma^\gd_t=(q^\gd_t,\gamma^\gd_t)$ for $t\in[0,T_\gd]$.

Now, by the Herculean Theorem (see \cite{sell2013dynamics}, Theorem 47.6), since $\cM^\gd$ is invariant, $\Gamma^\gd_t$ is in fact an element of $\mathbf{H}^{-r+2}_\theta$ and, by \cite{sell2013dynamics} Theorem 48.5, $\partial_t \Gamma^\gd_{s+t}=(\partial_t q^\gd_{s+t},\dot \gamma ^\gd_{s+t})$ is in $C([0,T_\gd),\mathbf{H}^{-r}_\theta)$ and is is solution to
\begin{equation}\label{eq:syst Gamma}
\left\{
\begin{array}{rl}
\partial_t \eta_t &= \cL \eta_t +\gd D G_1(\Gamma^\gd_{s+ t})[ \nu_{ t}]\\
\dot n_t &=\gd  D G_2( \Gamma^\gd_{ s+t})[ \nu_{ t}]
\end{array}
\right. ,
\end{equation}
which means in particular that $ \partial_{t} \Gamma^\gd_{s+t}=\Phi_{s+t,s}\partial_{t} \Gamma^\gd_{s}$. Now $\partial_t \Gamma^\gd_{s+t}$ is a periodic solution to \eqref{eq:syst Gamma}, and the same arguments induce that $\partial^2_t \Gamma^\gd_{s+t}$ is in $C([0,T_\gd),\mathbf{H}^{-r}_\theta)$.

In addition, it is proved in \cite{Bates:2008} that $\cM^\gd$ is foliated by $C^1$ invariant foliations: a neighborhood $\cW^\gd$ of $\cM^\gd$ satisfies the decomposition 
$\cW^\gd=\cup_{s\in [0,T_\gd)} \cW^\gd_s$, where  $\cW^\gd_s$ corresponds to the elements of $\mu\in \mathbf{H}^{-r}_\theta$ such that $T^{nT_\gd}(\mu)$ converges exponentially fast to  $\Gamma^\gd_s$ as $n$ goes to infinity. The projections $\Pi^{\gd,c}_s$ and $\Pi^{\gd,s}_s$ correspond then respectively to the projections on the tangent space to $\cM^\gd$ and to $\cW^\gd_s$ at $\Gamma^\gd_s$. The linear operator $\Phi^\gd_{s+t,s}=DT^t(\Gamma^\gd_s)$ commutes then with these projections, and is bounded from above and below in the direction of the tangent space to $\cM^\gd$, while it is contractive in the direction of the tangent space to stable foliations.

In addition to the contractive property, the regularization effect of $\Phi^\gd$ given in \eqref{eq:Phi contracts} is a consequence of the fact that $\Phi_{t+s,s}\nu=\nu_t$ where $\nu_0=\nu$ and $\nu_t=(\eta_t,n_t)$ is solution to
\begin{equation}\label{eq:syst Phi}
\left\{
\begin{array}{rl}
\partial_t \eta_t &= \cL \eta_t +\gd D G_1(\Gamma^\gd_{ s+t})[ \nu_{ t}]\\
\dot n_t &=\gd  D G_2( \Gamma^\gd_{ s+t})[ \nu_{ t}]
\end{array}
\right. .
\end{equation}
The operator $\tilde \cL(\eta,n)=(\cL \eta,0)$ is sectorial in $\mathbf{H}^{-r}_\theta$ and thus induces regularization properties for the solutions to \eqref{eq:syst Phi}, and thus for $\Phi^\gd$. More precisely we are in fact exactly in the situation of \cite{Henry:1981}, Theorem 7.2.3 and the following remark. Indeed, for $s\in [0,T_\gd)$ we can define the operator $U^\gd_s = \Phi^\gd_{s+T_\gd,s}$, and we can deduce from above spectral properties  for $U^\gd_t$. Since $\Gamma^\gd$ is a periodic solution $U^\gd_s$ admits $1$ as eigenvalue, with eigenfunction $\partial_s \Gamma^\gd_s$ and corresponding projection $\Pi^{\gd,c}_s$, and due to the contractive properties $\Phi^\gd$ the rest of the spectrum of $U^\gd_s$ is located in a disk centered at $0$ with radius $e^{-\lambda_{\gd} T_\gd}$. We can then apply Theorem 7.2.3 and the following remark to obtain \eqref{eq:Phi contracts} (reducing slightly the value of $\lambda_\gd$).

The $C^1$ regularity of $s\mapsto \Pi^{\gd,c}_s$ is not a direct consequence the normally hyperbolic results of \cite{Bates:2008} (they prove that $\cW^\gd_s$ has a Hölder regularity with respect to $s$), but since we are in the case of a periodic solution we have an explicit formula for $\Pi^{\gd,c}_s$: $1$ is an isolated eigenvalue of $U^\gd_t$, so for $\cC_\gep$ the circle centered at $1$ with radius $\gep>0$, with $\gep$ small enough, we have
\begin{equation}
\Pi^{\gd,c}_s =\frac{1}{2i\pi}\int_{\cC_\gep} (\lambda-U^\gd_s)^{-1}\dd\lambda. 
\end{equation}
But applying \cite{Henry:1981}, Theorem 3.4.4., $t\mapsto U^\gd_s$ is $C^1$, with $\partial_s U^\gd_s \zeta= \zeta_{T_\gd}=(\zeta^1_{T_\gd},\zeta^2_{T_\gd})$, where $\zeta_0=\zeta$ and
\begin{equation}
\left\{
\begin{array}{rl}
\partial_t \zeta^1_t &= \cL \zeta^1_t +\gd D G_1(\Gamma^\gd_{ s+t})[ \zeta_{ t}]+\gd D ^2G_1(\Gamma^\gd_{ s+t})[\partial_t \Gamma^\gd_{s+t}, \zeta_{ t}]\\
\dot \zeta^2_t &=\gd  D G_2( \Gamma^\gd_{s+ t})[ \zeta_{ t}]+\gd D ^2G_2(\Gamma^\gd_{ s+t})[\partial_t \Gamma^\gd_{s+t}, \zeta_{ t}].
\end{array}
\right. ,
\end{equation}
and thus $s\mapsto \Pi^{\gd,c}_s$ is also $C^1$.

It is not immediate that $q^\gd_s$ is a probability distribution, since we apply the results of \cite{Bates:2008} considering solutions $p_t \in H^{-r}_\theta$ satisfying $\int_{\bbR^d}p_t=1$ but without any hypotheses on non negativeness. However, $\tilde\cM^\gd$ is in the basin of attraction of $\cM^\gd$, so any $(q^\gd_s,m^\gd_s)\in \cM^\gd$ is the limit in $\mathbf{H}^{-r}_\theta$ of $(p_t,m_t)=T^t(\rho,\ga^\gd_u)$ for some $u\in [0,\frac{T_\ga}{\gd})$. So, since in this case $p_t$ is a probability distribution (recall that it is the probability distribution of $X_t-\bbE[X_t]$, where $X_t$ satisfies \eqref{eq:McKean} with initial distribution $\rho$), we deduce that $\langle q^\gd_s,\varphi\rangle \geq 0$ for any smooth function $\varphi$ with compact support, and thus $q^\gd_s$  is also a probability distribution.
\end{proof}

\section{Proof of Theorem~\ref{th:Theta}}\label{sec:proof theta}

As it was already explained in Section~\ref{sec:main results}, the existence of the map $\Theta^\gd$ is a consequence of the foliation property proved in \cite{Bates:2008}.
$\Theta^\gd$ satisfies moreover the relation
\begin{equation}
\Gamma^\gd_{\Theta(\mu)}= \lim_{n\rightarrow\infty} T^{nT_\gd}\mu.
\end{equation}
Our aim in the present section is to prove the $C^2$ regularity of $\Theta^\gd$. Following ideas of \cite{guckenheimer1975isochrons}, we will prove uniform in time bounds for the first and second derivatives of the flow $T^t$, which will induce the regularity of 
\begin{equation}
\label{eq:S}
S(\mu):=\lim_{n\rightarrow\infty} T^{nT_\gd}\mu
\end{equation} and thus the regularity of $\Theta^\gd$.

\medskip

\begin{proof}[Proof of Theorem~\ref{th:Theta}]~\\
{\bf Step 1:} Let us first show that for some constant $c_1>0$

\begin{equation}
\sup_{ t\geq 0} \sup_{ \mu \in \mathcal{ V} \left( \cM^\gd, \varepsilon\right)}\left\Vert DT^{ t}(\mu) \right\Vert_{\cB( \mathbf{ H}_{ \theta}^{ - r})} \leq c_1 
\end{equation}
where $\cV(\cM,\gep):= \left\lbrace \mu\in \mathbf{ H}_{ \theta}^{ -r},\ {\rm dist}_{ \mathbf{ H}_{ \theta}^{ -r}}\left(\mu, \mathcal{ M}\right)< \varepsilon\right\rbrace$ is a neighborhood of $\cM$ on which the trajectories are attracted to the cycle. For $ \mu_{ 0}=(p_0,m_0)\in \cV(\cM,\gep)$ and $u = \Theta( \mu_{ 0})$, we have, denoting by $ \nu_t=( \eta_t,n_t)=DT^t( \mu_{ 0})[ \nu_0]$,
\begin{equation}
 \nu_t = \Phi^\gd_{u+t,u} \nu_0+\gd \int_0^t \Phi^\gd_{u+t,u+s}\left( DG( \mu_{ s})-DG( \Gamma^\gd_{ u+s})\right)[ \nu_s]\dd s.
\end{equation}
Let us now prove that there exists a constant $C_G$ such that, for $ \mu=(p, m)$ and $ \Gamma=(q, \gamma)$,
\begin{equation}
\left\Vert DG( \mu)-DG( \Gamma)\right\Vert_{\cB\left(\mathbf{H}^{-r}_\theta,\mathbf{H}^{-(r+1)}_\theta\right)}\leq C_G \left\Vert \mu - \Gamma \right\Vert_{\mathbf{H}^{-r}_\theta}.
\end{equation}
We have, for $ \nu=(\eta, n)$
\begin{align}
\left(DG_1(\mu)-DG_1(\Gamma)\right)[ \nu]
=&
-\nabla\cdot \left( \eta(F_m-F_\gamma)\right)+\nabla\cdot \left( \eta\left(\int F_m p-\int F_\gamma q\right)\right)\\
&-\nabla\cdot \left(pD F_m [n]-qDF_\gamma[n]\right)\\
&+\nabla\cdot \left(p\int F_m \eta-q\int F_\gamma \eta\right) \\
&+\nabla\cdot \left(p\int D F_m[n] p-q\int DF_\gamma[n] q\right),
\end{align}
and
\begin{equation}
\left(DG_2(\mu)-DG_2( \Gamma)\right)[ \nu]
=\int (F_m-F_\gamma) \eta+ \int DF_m[n]p-\int DF_\gamma[n]q.
\end{equation}
For the first term, we obtain
\begin{equation}
\left\Vert \nabla\cdot \left( \eta(F_m-F_\gamma)\right)\right\Vert_{H^{-(r+1)}_\theta}\leq C_1 \left\Vert \eta(F_m-F_\gamma)\right\Vert_{H^{-r}_\theta},
\end{equation}
and since, for $f\in H^r_\theta$,
\begin{equation}
\langle \eta(F_m-F_\gamma),f\rangle \leq \Vert \eta\Vert_{H^{-r}_\theta} \Vert (F_m-F_\gamma)f\Vert_{H^r_\theta}\leq C_2|m-\gamma| \Vert \eta\Vert_{H^{-r}_\theta} \Vert f\Vert_{H^r_\theta},
\end{equation}
where we have used the fact that all the derivatives of $F$ are Lipschitz, we get, for some $C_3>0$,
\begin{equation}
\left\Vert \nabla\cdot \left(\eta(F_m-F_\gamma)\right)\right\Vert_{H^{-(r+1)}_\theta}\leq C_3|m-\gamma|  \Vert \eta\Vert_{H^{-r}_\theta} .
\end{equation}
For the second term, since
\begin{equation}
\left\vert \int F_m p-\int F_\gamma q\right\vert \leq \left\vert \int F_m (p-q)\right\vert +\left\vert\int (F_m-F_\gamma) q\right\vert \leq C_4 \left(\Vert p-q\Vert_{H^{-r}_\theta}+|m-\gamma|\right),
\end{equation}
we have
\begin{equation}
\left\Vert \nabla\cdot \left( \eta\left(\int F_m p-\int F_\gamma q\right)\right)\right\Vert_{H^{-(r+1)}_\theta}\leq C_5 \left(\Vert p-q\Vert_{H^{-r}_\theta}+|m-\gamma|\right).
\end{equation}
The other terms can be tackled in a similar way. Now, since $\mu_0\in \cW^\gd_u$, we have for some $C_{\Gamma^\gd}>0$,
\begin{equation}\label{eq:p m cv to q gamma}
\left\Vert \mu_{ s}- \Gamma^\gd_{u+s}\right\Vert_{\mathbf{H}^{-r}_\theta}\leq C_{\Gamma^\gd} e^{-\lambda_{ \delta} s} \left\Vert \mu_{ 0}-  \Gamma^\gd_{ u}\right\Vert_{\mathbf{H}^{-r}_\theta},
\end{equation}
and from the estimates obtained above , we deduce
\begin{equation}
\label{aux:nut_Gronwall}
\Vert \nu_t\Vert_{\mathbf{H}^{-r}_\theta}\leq C_6\Vert \nu_0\Vert_{\mathbf{H}^{-r}_\theta}
+C _6\gd \int_0^t \left( 1+(t-s)^{ -\frac{ 1}{ 2}} e^{-\lambda_{ \delta}(t-s)}\right) e^{-\lambda_{ \delta} s}\Vert \nu_s\Vert_{\mathbf{H}^{-r}_\theta}\dd s.
\end{equation}
Applying Lemma~\ref{lem:Gronwall} for $ \phi(u)= u^{ - \frac{ 1}{ 2}} e^{ - \lambda_{ \delta} u}$, we obtain from \eqref{eq:Gronwall} that
\begin{equation}
\label{eq:uniform_control_nut}
\sup_{ t\geq 0} \left\Vert \nu_{ t} \right\Vert_{  \mathbf{ H}_{ \theta}^{ -r}}\leq c_1\left\Vert \nu_{ 0} \right\Vert_{ \mathbf{ H}_{ \theta}^{ -r}}
\end{equation} for some $c_1>0$.

\bigskip

\noindent
{\bf Step 2:} let us now show that $\left( DT^{nT_\gd}\right)_{n\geq 0}$ in the space $C\left(\cV(\cL^\gd,\gep),\cB\left(\mathbf{H}^{-r}_\theta\right)\right)$ is a Cauchy sequence, which implies that $ \mu\mapsto S(\mu)$ is $C^1$ (recall \eqref{eq:S}).

\medskip

For $n\geq m$ we have
\begin{align}
\nu_{nT_\gd}- \nu_{mT_\gd} =&\left( \Phi^\gd_{u+nT_\gd, u}-\Phi^\gd_{u+mT_\gd, u}\right) \nu_0\\
&+\gd \int_{mT_\gd}^{nT_\gd} \Phi^\gd_{u+nT_\gd,u+s}\left( DG( \mu_{ s})-DG( \Gamma^\gd_{ u+s})\right)[ \nu_s]\dd s\nonumber\\
&+\gd \int_{0}^{mT_\gd}\left( \Phi^\gd_{u+nT_\gd,u+s}-\Phi^\gd_{u+mT_\gd,u+s}\right)\left( DG(\mu_{ s})-DG( \Gamma^\gd_{ u+s})\right)[ \nu_s]\dd s\nonumber.
\end{align}
For the first term, we get
\begin{align}
\left\Vert \left( \Phi_{u+nT_\gd, u}-\Phi_{u+mT_\gd, u}\right) \nu_0\right\Vert_{\mathbf{H}^{-r}_\theta} &=\left\Vert \left( \Phi_{u+nT_\gd, u}-\Phi_{u+mT_\gd, u}\right) \Pi_{ \delta, u} \nu_0\right\Vert_{\mathbf{H}^{-r}_{ \theta}}\nonumber\\
&\leq C_7e^{-\lambda_{ \delta} mT_\gd}\left\Vert \nu_0\right\Vert_{\mathbf{H}^{-r}_\theta}.
\end{align}
For the second one, using \eqref{eq:uniform_control_nut},
\begin{align}
\bigg\Vert   &\int_{mT_\gd}^{nT_\gd} \Phi^\gd_{u+nT_\gd,u+s}\big( DG( \mu_{ s})-DG( \Gamma^\gd_{ u+s})\big)[ \nu_s]\dd s  \bigg\Vert_{\mathbf{H}^{-r}_\theta}\nonumber\\
&\leq  C_8  \left\Vert \mu- \Gamma^\gd_u\right\Vert_{\mathbf{H}^{-r}_\theta}\Vert \nu_0\Vert_{\mathbf{H}^{-r}_\theta}\int_{mT_\gd}^{nT_\gd} \left( 1+(nT_\gd-s)^{ - \frac{ 1}{ 2}} e^{-\lambda_{ \delta}(nT_\gd-s)}\right) e^{-\lambda_{ \delta} s}\dd s\nonumber\\
&=\frac{ C_8 \left\Vert \mu- \Gamma^\gd_u\right\Vert_{\mathbf{H}^{-r}_\theta}\left\Vert \nu_{ 0} \right\Vert_{ \mathbf{ H}_{ \theta}^{ - r}}}{ \lambda_{ \delta}}e^{-\lambda_{ \delta} m T_{ \delta}} \left( 1 + e^{ - \lambda_{ \delta} (n-m) T_{ \delta}} \left(2 \lambda_{ \delta} \sqrt{ (n-m)T_{ \delta}}-1\right)\right)\nonumber\\
&\leq C_9  \left\Vert \mu- \Gamma^\gd_u\right\Vert_{\mathbf{H}^{-r}_\theta}\left\Vert \nu_{ 0} \right\Vert_{ \mathbf{ H}_{ \theta}^{ - r}}e^{-\lambda_{ \delta} m T_{ \delta}}.
\end{align}
For the last term, remark first that
\begin{equation}
\Phi^\gd_{u+nT_\gd,u+s}-\Phi^\gd_{u+mT_\gd,u+s} =\left(\Phi^\gd_{u+nT_\gd,u+mT}-I_d\right)\Pi^{\gd,s}_{u+mT_\gd}\Phi^\gd_{u+mT_\gd,u+s},
\end{equation}
so that, using again \eqref{eq:uniform_control_nut},
\begin{align}
\bigg\Vert  \int_{0}^{mT_\gd}\big( \Phi^\gd_{u+nT_\gd,u+s}-&\Phi^\gd_{u+mT_\gd,u+s}\big)\left( DG( \mu_{ s})-DG( \Gamma_{ u+s})\right)[ \nu_s]ds  \bigg\Vert_{\mathbf{H}^{-r}_\theta}\nonumber\\
&\leq C_{10} \left\Vert \mu- \Gamma^\gd_u\right\Vert_{\mathbf{H}^{-r}_\theta}\Vert \nu_0\Vert_{\mathbf{H}^{-r}_\theta}\int_{0}^{mT_\gd}  (mT_\gd-s)^{ - \frac{ 1}{ 2}} e^{-\lambda_{ \delta}(mT_\gd-s)} e^{-\lambda_{ \delta} s}ds\nonumber\\
&= 2C_{10}  \left\Vert \mu- \Gamma^\gd_u\right\Vert_{\mathbf{H}^{-r}_\theta}\left\Vert \nu_{ 0} \right\Vert_{ \mathbf{ H}_{ \theta}^{ -r}} \sqrt{ m T_{ \delta}} e^{ - \lambda_{ \delta} m T_{ \delta}}.
\end{align}
Since the constants above are uniform in $\mu\in \cV$, we deduce that $\left( DT^{nT_\gd}\right)_{n\geq 0}$ is indeed a Cauchy sequence. Thus $S$ is $C^1$ with $DS(\mu)=\lim_{n\rightarrow\infty} DT^{nT_\gd}(\mu)$. 

\medskip

Before moving to the second derivative, let us have a closer look at $DS$. We have
\begin{align}
\left\Vert  \Pi^{\gd,s}_{u+nT_\gd} \nu_{ nT_{ \delta}}    \right\Vert_{\mathbf{H}^{-r}_\theta} \leq & 
\left\Vert  \Pi^{\gd,s}_{u+nT_\gd} \Phi^\gd_{ u+ n T_{ \delta}, u}\nu_0     \right\Vert_{\mathbf{H}^{-r}_\theta} \nonumber\\
& +\left\Vert   \int_{0}^{nT_\gd}\Pi^{\gd,s}_{u+nT_\gd} \Phi^\gd_{u+nT_\gd,u+s}\left( DG(\mu_{ s})-DG( \Gamma^\gd_{ u+s})\right)[ \nu_s]ds    \right\Vert_{\mathbf{H}^{-r}_\theta},
\end{align}
and we can bound the right hand side in three steps. Firstly,
\begin{equation}
\left\Vert  \Pi^{\gd,s}_{u+nT_\gd} \Phi^\gd_{ u+ n T_{ \delta}, u}\nu_0     \right\Vert_{\mathbf{H}^{-r}_\theta} \leq C_{\Phi,\gd} e^{-\lambda_{ \delta} nT_\gd}\left\Vert  \nu_0     \right\Vert_{\mathbf{H}^{-r}_\theta}.
\end{equation}
Secondly, since $\sup_{ \mu\in \cV}\Vert DG( \mu)\Vert_{\cB\left(\mathbf{H}^{-r}_\theta,\mathbf{H}^{-(r+1)}_\theta\right)}\leq C_G$, 
\begin{align}
\Bigg\Vert   \int_{0}^{\frac{nT_\gd}{2}} \Pi^\gd_{u+nT_\gd}&\Phi^\gd_{u+nT_\gd,u+s}\left( DG( \mu_{ s})-DG( \Gamma^\gd_{ u+s})\right)[ \nu_s]\dd s    \Bigg\Vert_{\mathbf{H}^{-r}_\theta}\nonumber\\
&\leq C_{11}  \left\Vert \mu- \Gamma^\gd_u\right\Vert_{\mathbf{H}^{-r}_\theta}\left\Vert \nu_0\right\Vert_{\mathbf{H}^{-r}_\theta}\int_0^{\frac{nT_\gd}{2}} \left(nT_\gd-s\right)^{-\frac12}e^{-\lambda_{ \delta}(nT_\gd-s)}\dd s \nonumber\\
&\leq C_{12}  \left\Vert \mu- \Gamma^\gd_u\right\Vert_{\mathbf{H}^{-r}_\theta}n^\frac12e^{-\frac{ \lambda_{ \delta} nT_\gd}{2}}\left\Vert \nu_0\right\Vert_{\mathbf{H}^{-r}_\theta}.
\end{align}
Thirdly, by similar arguments as above (replacing $mT_\gd$ with $n\frac{T_\gd}{2}$),
\begin{multline}
\left\Vert   \int_{\frac{nT_\gd}{2}}^{nT_\gd} \Phi^\gd_{u+nT_\gd,u+s}\left( DG( \mu_{ s})-DG( \Gamma^\gd_{ u+s})\right)[ \nu_s]\dd s    \right\Vert_{\mathbf{H}^{-r}_\theta} \\
\leq C_{13}  \left\Vert \mu- \Gamma^\gd_u\right\Vert_{\mathbf{H}^{-r}_\theta}e^{-\lambda_{ \delta} n\frac{T_\gd}{2}} \Vert  \nu_0\Vert_{\mathbf{H}^{-r}_\theta}.
\end{multline}
We deduce that $\Pi_{\Theta(\mu)}DS(\mu)=0$, so that $DS$ has rank $1$ and thus there exists a family of linear forms $l_{ \mu} \in \cB\left(\mathbf{H}^{-r}_\theta,\bbR \right)$ (that depend continuously on $ \mu$) such that, for $u=\Theta(\mu)$,
\begin{equation}
DS(\mu)[ \nu] = l_{\mu}[ \nu]\partial_u \Gamma_{ u},
\end{equation} 
and we have proved, for $ \nu_t=DT^t(\mu)[ \nu_0]$,
\begin{equation}
\left\Vert  \nu_{nT_\gd}-  l_{ \mu}[ \nu_0] \partial_u \Gamma_{ u} \right\Vert_{\mathbf{H}^{-r}_\theta} \leq C_{13}n^\frac12 e^{-\lambda_{ \delta} n \frac{T_\gd}{2}}
\left\Vert  \nu_{0} \right\Vert_{\mathbf{H}^{-r}_\theta}.
\end{equation}
With similar computations one can in fact show that
\begin{equation}
\label{eq:nut_close_DS}
\left\Vert  \nu_{t}-  l_{ \mu}[ \nu_0] \partial_u \Gamma_{ u+t}\right\Vert_{\mathbf{H}^{-r}_\theta} \leq C_{14} t^\frac12 e^{-\lambda_{ \delta}  \frac{t}{2}}
\left\Vert  \nu_{0} \right\Vert_{\mathbf{H}^{-r}_\theta}.
\end{equation}
In the case of $ \mu= \Gamma_{ u}^{ \delta}$, we deduce in particular that
\begin{equation}
DS(\Gamma^\gd_u) =\Pi^{\gd,c}_u.
\end{equation}

In fact, we have proved a more precise estimate: if $ \nu^2_t=DT^t(\mu)[ \nu_0]$, $\nu^1_t=DT^t(\Gamma^\gd_u)[ \nu_0] $ with $u=\Theta(\mu)$, the estimates above lead to
\begin{equation}\label{eq:delta nut}
\left\Vert  \nu^2_{t}-\nu^1_t-  \left(l_{ \mu}[ \nu_0]-l_{\Gamma^\gd_u}[ \nu_0] \right)\partial_u \Gamma_{ u+t}\right\Vert_{\mathbf{H}^{-r}_\theta} \leq C_{15} \left\Vert \mu-\Gamma^\gd_u\right\Vert_{\mathbf{H}^{-r}_\theta} t^\frac12 e^{-\lambda_{ \delta}  \frac{t}{2}}
 \left\Vert  \nu_{0} \right\Vert_{\mathbf{H}^{-r}_\theta}.
\end{equation}

\bigskip

\noindent
{\bf Step 3:} let us now show that for a constant $c_2>0$,
\begin{equation}
\sup_{ t\geq 0} \sup_{ \mu \in \mathcal{ V} \left( \Gamma^{ \delta}, \varepsilon\right)}\left\Vert D^{ 2}T^{ t}(\mu) \right\Vert_{\cB\cL(\mathbf{H}^{-r}_\theta)} \leq c_2. \label{eq:bound D2T}
\end{equation}

From \eqref{eq:dxi}, we deduce, for $\xi_t=D^2 T^t( \mu)[ \nu,w]$, the following mild formulation (recall that $\xi_0=0$):
\begin{equation}\label{eq:var const eta}
\xi_t = \gd\int_0^t \Phi^\gd_{u+t,u+s}\left(D^2G( \mu_{ s})[ \nu_s,w_s]+\left(DG( \mu_{ s})-DG( \Gamma^\gd_{ u+s})\right) \xi_s \right)\dd s,
\end{equation}
where $ \nu_t=DT^t( \mu_{ 0})[ \nu]$, $w_t=DT^t( \mu_{ 0})[w]$.
With similar arguments as above, we obtain
\begin{align}
\Bigg\Vert \int_0^t& \Phi^\gd_{u+t,u+s}\left(D^2G( \mu_{ s})[ \nu_s,w_s]-D^2G( \Gamma^\gd_{ u+s})[ \nu_s,w_s]\right)\dd s\Bigg\Vert_{\mathbf{H}^{-r}_\theta}\nonumber\\
&\leq C_{16} \int_0^t \left(1+(t-s)^{ - \frac{ 1}{ 2}}e^{-\lambda_{ \delta}(t-s)}\right)\left\Vert  \mu_{ s}- \Gamma^\gd_{ u+s}\right\Vert_{\mathbf{H}^{-r}_\theta} \left\Vert \nu_s \right\Vert_{\mathbf{H}^{-r}_\theta}  \left\Vert  w_s\right\Vert_{\mathbf{H}^{-r}_\theta}\dd s\nonumber\\
&\leq C_{16}\left\Vert \nu_0 \right\Vert_{\mathbf{H}^{-r}_\theta}  \left\Vert  w_0\right\Vert_{\mathbf{H}^{-r}_\theta}.
\end{align}
Remark now that
\begin{equation}
\partial^2_t\Gamma^\gd_{u+ t}  = (\cL \partial_t q^\gd_{u+t},0)+\gd DG( \Gamma^\gd_{u+ t})[\partial_t \Gamma_{u+ t}],
\end{equation}
and
\begin{equation}
\partial^3_t \Gamma^\gd_{u+ t} = \left(\cL \partial^2_t q^\gd_{u+t},0\right)+\gd DG( \Gamma^\gd_{u+ t})[\partial^2_t \Gamma^\gd_{u+ t}]+\gd D^2G( \Gamma^\gd_{u+ t})[\partial_t \Gamma^\gd_{u+ t},\partial_t \Gamma^\gd_{u+ t}],
\end{equation}
and thus
\begin{equation}
\label{aux:diffGamma}
\partial^2_t \Gamma^\gd_{ u+t} = \Phi^\gd_{u+t,u}\partial^2_t \Gamma^\gd_{ u}
+\gd \int_0^t \Phi^\gd_{u+t,u+s}D^2G( \Gamma^\gd_{ u+s})[\partial_s \Gamma^\gd_{ u+s},\partial_s \Gamma^\gd_{ u+s}]ds.
\end{equation}
So, in particular, since 
\begin{equation}
\Pi^{\gd,c}_u\Phi^\gd_{u+T_\gd,u}\partial^2_u \Gamma_{ u}   = \Pi^{\gd,c}_u(\partial^2_u \Gamma_{ u}),
\end{equation}
we deduce from \eqref{aux:diffGamma} that
\begin{equation}\label{eq: proj int q =0}
\Pi^{\gd,c}_u\left(\int_0^{T_\gd} \Phi^\gd_{u+T_\gd,u+s}D^2G( \Gamma^\gd_{ u+s})[\partial_s \Gamma^\gd_{ u+s},\partial_s \Gamma^\gd_{ u+s}]ds\right)=0.
\end{equation}
Now, recalling \eqref{eq:nut_close_DS},
\begin{align}\label{eq:bound zt}
\left\Vert  \nu_{t}-  l_{ \mu}[ \nu_0] \partial_u \Gamma_{ u+t}   \right\Vert_{\mathbf{H}^{-r}_\theta}&\leq C_{14}t^\frac12 e^{-\lambda_{ \delta}  \frac{t}{2}}\left\Vert  \nu_{0} \right\Vert_{\mathbf{H}^{-r}_\theta},\\\label{eq:bound wt}
\left\Vert  w_{t}-  l_{ \mu}[w_0] \partial_u \Gamma_{ u+t}   \right\Vert_{\mathbf{H}^{-r}_\theta}& \leq C_{14} t^\frac12 e^{-\lambda_{ \delta}  \frac{t}{2}}\left\Vert  w_{0} \right\Vert_{\mathbf{H}^{-r}_\theta},
\end{align}
and we deduce
\begin{multline}
\Bigg\Vert\int_0^t \Phi_{u+t,u+s}D^2G(\mu_{s})[ \nu_s,w_s] ds\\ - l_{ \mu}[ \nu_0] l_{ \mu}[w_0] \int_0^t \Phi_{u+t,u+s}D^2G( \mu_{s})[\partial_u \Gamma_{ u+s},\partial_u \Gamma_{ u+s}] \dd s\Bigg\Vert_{\mathbf{H}^{-r}_\theta}\\
\leq C_{17} \left\Vert  \nu_{0} \right\Vert_{\mathbf{H}^{-r}_\theta}\left\Vert  w_{0} \right\Vert_{\mathbf{H}^{-r}_\theta}.
\end{multline}
So, recalling \eqref{eq: proj int q =0}, and since
\begin{multline}
\Bigg\Vert  \Pi^{\gd,s}_{u+t} \int_0^t \Phi^\gd_{u+t,u+s}D^2G( \mu_{ s})[\partial_u \Gamma^\gd_{ u+s}, \partial_u \Gamma^\gd_{ u+s}] \dd s\Bigg\Vert_{\mathbf{H}^{-r}_\theta} \\\
\leq C_{18} \int_0^t (t-s)^{ - \frac{ 1}{ 2}} e^{-\lambda_{ \delta} (t-s)}ds\leq C_{19},
\end{multline}
we deduce, coming back to \eqref{eq:var const eta}, that
\begin{multline}
\left\Vert \xi_t\right\Vert_{\mathbf{H}^{-r}_\theta}\leq C_{19}\left\Vert  \nu_{0} \right\Vert_{\mathbf{H}^{-r}_\theta}\left\Vert  w_{0} \right\Vert_{\mathbf{H}^{-r}_\theta}+\gd\left\Vert\int_0^t \Phi^\gd_{u+t,u+s}\left(DG( \mu_{ s})-DG( \Gamma^\gd_{ u+s})\right) \xi_s  {\rm d}s\right\Vert_{\mathbf{H}^{-r}_\theta}.
\end{multline}
Relying again on \eqref{eq:p m cv to q gamma}, we deduce that, for some $c_2>0$,
\begin{equation}
\left\Vert \xi_{ t}\right\Vert_{\mathbf{H}^{-r}_\theta}\leq c_2\left\Vert  \nu_{0} \right\Vert_{\mathbf{H}^{-r}_\theta}\left\Vert  w_{0} \right\Vert_{\mathbf{H}^{-r}_\theta},
\end{equation}
which implies \eqref{eq:bound D2T}.

\bigskip

{\bf Step 4:} let us now prove that $\left(D^2 T^{nT_\gd}\right)_{n\geq 0}$ in the space $C\left(\cV(\cM^\gd,\gep),\cB\cL\left(\mathbf{H}^{-r}_\theta  \right)\right)$ is a Cauchy sequence, which implies that $\mu\mapsto S( \mu)$ is $C^2$.

\medskip

We have, for $n\geq m$,
\begin{align}\label{eq:decomp etan-etam}
\xi_{nT_\gd}-\xi_{mT_\gd} =& \int_{0}^{nT_\gd}\Phi^\gd_{u+nT_\gd,u+s}D^2G( \Gamma ^\gd_{ u+s})[ \nu_s,w_s]\dd s\\
&-\int_{0}^{mT_\gd}\Phi^\gd_{u+mT_\gd,u+s}D^2G( \Gamma^\gd_{ u+s})[ \nu_s,w_s]\dd s\nonumber\\
&+ \int_{mT_\gd}^{nT_\gd}\Phi^\gd_{u+nT_\gd,u+s}\left(D^2G( \mu_{ s})[ \nu_s,w_s]-D^2G( \Gamma^\gd_{u+s})[ \nu_s,w_s]\right)\dd s\nonumber\\
&+ \int_{0}^{mT_\gd}\left(\Phi^\gd_{u+nT_\gd,u+s}-\Phi^\gd_{u+mT_\gd,u+s}\right)\nonumber\\
&\qquad \qquad\qquad\qquad\qquad\times \left(D^2G( \mu_{ s})[ \nu_s,w_s]-D^2G( \Gamma^\gd_{u+s})[ \nu_s,w_s]\right)\dd s\nonumber\\
&+\int_{mT_\gd}^{nT_\gd}\Phi^\gd_{u+nT_\gd,u+s}\left(DG( \mu_{ s})-DG( \Gamma^\gd_{u+s})\right)\xi_s \dd s\nonumber\\
&+\int_{0}^{nT_\gd}\left(\Phi^\gd_{u+nT_\gd,u+s}-\Phi^\gd_{u+mT_\gd,u+s}\right)\left(DG( \mu_{ s})-DG(\Gamma^\gd_{u+s})\right)\xi_s \dd s.\nonumber
\end{align}
Let us define
\begin{align}
R^{norm}_n&:=\int_{0}^{nT_\gd}\Pi^{\gd,s}_{u+nT_\gd}\Phi^\gd_{u+nT_\gd,u+s}D^2G( \Gamma^\gd_{u+s})[\partial_s \Gamma^\gd_{ u+s},\partial_s \Gamma^\gd_{u+s}]\dd s\\
&\,=\sum_{j=0}^{n-1}\left(\Phi^\gd_{u+T_\gd,u}\Pi^{\gd,s}_u\right)^j\int_{0}^{T_\gd}\Phi^\gd_{u+T_\gd,u+s}D^2G( \Gamma^\gd_{u+s})[\partial_s \Gamma^\gd_{u+s},\partial_s \Gamma^\gd_{u+s}]\dd s\nonumber,
\end{align}
and
\begin{align}
R^{tang}_n[ \nu_0,w_0] &:= \int_{0}^{nT_\gd}\Pi^{\gd,c}_{u+nT_\gd}\Phi^\gd_{u+nT_\gd,u+s}D^2G( \Gamma^\gd_{u+s})[ \nu_s,w_s]\dd s\\
&\, =\sum_{j=0}^{n-1}\int_0^{T_\gd}\Pi^{\gd,c}_{u+T_\gd}\Phi^\gd_{u+T_\gd,u+s}D^2G( \Gamma^\gd_{u+s})[ \nu_{jT_\gd+s},w_{jT_\gd+s}]\dd s.
\end{align}
It is clear that
\begin{equation}
\left\Vert R^{norm}_n-R^{norm}_m \right\Vert_{\mathbf{H}^{-r}_\theta}\leq C e^{-\lambda_{ \delta} mT_\gd}.
\end{equation}
Now, for $j\geq 1$, recalling \eqref{eq: proj int q =0}, \eqref{eq:bound zt} and \eqref{eq:bound wt} we have 
\begin{align}
\bigg\Vert     \int_0^{T_\gd}&\Pi^{\gd,c}_{u+T_\gd}\Phi^\gd_{u+T_\gd,u+s}D^2G( \Gamma^\gd_{u+s})[ \nu_{jT_\gd+s},w_{jT_\gd+s}]\dd s      \bigg\Vert_{\mathbf{H}^{-r}_\theta}\nonumber\\
&\leq C_{20}\left\Vert  \nu_{0} \right\Vert_{\mathbf{H}^{-r}_\theta}\left\Vert  w_{0} \right\Vert_{\mathbf{H}^{-r}_\theta} \int_0^{T_\gd}\left(1+(T_\gd-s)^{ - \frac{ 1}{ 2}} e^{-\lambda_{ \delta}(T_\gd-s)}\right)(jT_\gd+s)^\frac12e^{-\lambda_{ \delta} \frac{jT_\gd+s}{2}}\dd s\nonumber\\
&\leq C_{21}\left\Vert  \nu_{0} \right\Vert_{\mathbf{H}^{-r}_\theta}\left\Vert  w_{0} \right\Vert_{\mathbf{H}^{-r}_\theta} (j+1)^\frac12 e^{-\lambda_{ \delta} j \frac{T_\gd}{2}},
\end{align}
so that
\begin{equation}
\left\Vert R^{tan}_n[ \nu_0,w_0]-R^{tang}_m[ \nu_0,w_0] \right\Vert_{\mathbf{H}^{-r}_\theta}\leq C_{22} e^{-\lambda_{ \delta} m\frac{T_\gd}{4}}\left\Vert  \nu_{0} \right\Vert_{\mathbf{H}^{-r}_\theta}\left\Vert  w_{0} \right\Vert_{\mathbf{H}^{-r}_\theta}.
\end{equation}
Using similar argument as above, relying on \eqref{eq:bound zt} and \eqref{eq:bound wt},	
\begin{align}
\bigg\Vert \int_{0}^{nT_\gd}&\Phi^\gd_{u+nT_\gd,u+s}D^2G( \Gamma^\gd_{u+s})[ \nu_s,w_s]ds-R^{tan}_n[ \nu_0,w_0]-l_{ \mu}[ \nu_0] l_{ \mu}[w_0]R_n^{norm} \dd s\bigg\Vert_{\mathbf{H}^{-r}_\theta}\nonumber\\
&\leq C_{23}\left\Vert  \nu_{0} \right\Vert_{\mathbf{H}^{-r}_\theta}\left\Vert  w_{0} \right\Vert_{\mathbf{H}^{-r}_\theta}\int_{0}^{nT_\gd} (nT_\gd-s)^{ - \frac{ 1}{ 2}}e^{-\lambda_{ \delta}(nT_\gd-s)}s^\frac12 e^{-\lambda_{ \delta}\frac{s}{2}}\dd s \nonumber\\
&\leq C_{24} n^{ \frac{ 1}{ 2}} e^{-\lambda_{ \delta} n\frac{T_\gd}{2}}\left\Vert  \nu_{0} \right\Vert_{\mathbf{H}^{-r}_\theta}\left\Vert  w_{0} \right\Vert_{\mathbf{H}^{-r}_\theta}.
\end{align}
With all these estimates we are able to tackle the first two lines of \eqref{eq:decomp etan-etam}:
\begin{align}
\bigg\Vert \int_{0}^{nT_\gd}\Phi^\gd_{u+nT_\gd,u+s}&D^2G( \Gamma^\gd_{u+s})[ \nu_s,w_s]\dd s -\int_{0}^{mT_\gd}\Phi^\gd_{u+mT_\gd,u+s}D^2G( \Gamma^\gd_{u+s})[ \nu_s,w_s]\dd s\bigg\Vert_{\mathbf{H}^{-r}_\theta}\nonumber\\
& \leq C_{25} n^{ \frac{ 1}{ 2}} e^{-\lambda_{ \delta} n\frac{T_\gd}{4}}\left\Vert  \nu_{0} \right\Vert_{\mathbf{H}^{-r}_\theta}\left\Vert  w_{0} \right\Vert_{\mathbf{H}^{-r}_\theta}.
\end{align}
The other terms can be treated in a straightforward way, with similar estimates as the ones used in Step 2 and Step 3. At the end, one obtains
\begin{equation}
\left\Vert  \xi_{nT_\gd}- \xi_{mT_\gd} \right\Vert_{\mathbf{H}^{-r}_\theta}\leq C_{26} m^{ \frac{ 1}{ 2}} e^{-\lambda_{ \delta} m\frac{T_\gd}{4}},
\end{equation}
with a constant $C_{25}$ uniform in $\cV$. Hence, $ \mu\mapsto S(\mu)$ is thus $C^2$. Remark that we have in particular
\begin{equation}
D^2S( \Gamma_u)[ \nu,w] = \lim_{n\rightarrow\infty}\int_0^{nT_\gd} \Phi_{u+nT_\gd,u+s}D^2G( \Gamma_{u+s})[\Phi_{u+s,u} \nu,\Phi_{u+s,u}w]\dd s.
\end{equation}

\bigskip

{\bf Step 5:} from the previous steps, and the fact that $t \mapsto \Gamma^\gd_t$ is a $C^2$ bijection from $\bbR/T_\gd\bbZ$ to $\cM^\gd$ implies that $\Theta^\gd$ is itself $C^2$.

\medskip

For the last estimate of the Theorem, let us denote $\xi^2_t=D^2 T^t( \mu)[ \nu,w]$,  $\xi^1_t=D^2 T^t(\Gamma^\gd_u)[ \nu,w]$, $\nu^2_t=D T^t( \mu)[ \nu]$, $\nu^1_t=D T^t( \Gamma^\gd_u)[ \nu]$, $w^2_t=D T^t( \mu)[ w]$ and $w^1_t=D T^t( \Gamma^\gd_u)[ w]$. We then have the decomposition
\begin{align}
\xi^2_t-\xi^1_t =& \gd\int_0^t \Phi^\gd_{u+t,u+s}\left(DG(\mu_s)-DG\left(\Gamma^\gd_{u+s}\right) \right)\xi^2_s \dd s\\
&+\gd \int_0^t \Phi^\gd_{u+t,u+s}\left(D^2G(\mu_s)-D^2G\left(\Gamma^\gd_{u+s}\right) \right)[\nu^2_s,w^2_s] \dd s\nonumber\\
&+\gd \int_0^t \Phi^\gd_{u+t,u+s}D^2G\left(\Gamma^\gd_{u+s}\right) [\nu^2_s-\nu^1_s,w^2_s] \dd s\nonumber\\
&+\gd \int_0^t \Phi^\gd_{u+t,u+s}D^2G\left(\Gamma^\gd_{u+s}\right) [\nu^1_s,w^2_s-w^1_s] \dd s\nonumber.
\end{align}
Following similar estimates as in the previous steps, relying in particular on \eqref{eq:delta nut}, we obtain
\begin{equation}
\left\Vert \xi^2_t-\xi^1_t\right\Vert_{\mathbf{H}^{-r}_\theta} \leq C_{27}\left\Vert \mu-\Gamma^\gd_u\right\Vert_{\mathbf{H}^{-r}_\theta} \left\Vert \nu\right\Vert_{\mathbf{H}^{-r}_\theta} \left\Vert w\right\Vert_{\mathbf{H}^{-r}_\theta} ,
\end{equation}
which implies indeed that 
$
\left\Vert D^2\Theta^\gd(\mu) - D^2\Theta^\gd\left(\Gamma^\gd_u\right)\right\Vert_{\cB\cL (\mathbf{H}^{-r}_\theta)}\leq C_{28} \left\Vert \mu - \Gamma^\gd_{u}\right\Vert_{\mathbf{H}^{-r}_\theta}.
$
\end{proof}

\appendix

\section{Ornstein-Uhlenbeck operators}\label{app:OU}

We first prove the following lemma, which shows the link between the norm $\Vert f\Vert_{H^r_\theta}$ and the space derivatives.
\begin{lemma}\label{lem:derivative in Hr+1}
For all $ \theta>0$, there exists explicit positive constants $C_1$, $C_2$ such that for all $r\geq 0$:
\begin{equation}
\label{eq:equiv_norms}
C_1 \left(\Vert u\Vert_{H^{r}_{ \theta}}^2+\sum_{i=1}^d \left\Vert \partial_{x_i} u\right\Vert_{H^{r}_{ \theta}}^2\right)\leq \left\Vert u\right\Vert_{H^{r+1}_{ \theta}}^2\leq C_2 \left(\Vert u\Vert_{H^{r}_{ \theta}}^2+\sum_{i=1}^d \left\Vert \partial_{x_i} u\right\Vert_{H^{r}_{ \theta}}^2\right).
\end{equation}
\end{lemma}

\begin{proof}
For $u$ with decomposition $u=\sum_{l\in \bbN^d} u_l \, \psi_l$, we have $\partial_{x_i} u=\sum_{l\in \bbN^d} u_l \, \partial_{x_i} \psi_l$, and straightforward calculations (recalling \eqref{eq:decomp Ltheta} and using the fact that $h'_n(x)=\sqrt{n}h_{n-1}(x)$) show that
\begin{equation}
\partial_{ x_{ i}} \psi_l = \sqrt{l_i}\sqrt{\frac{ \theta k_i}{\gs_i^2}} \psi_{\check l_i} \mathbf{ 1}_{ l_{ i}\geq1},
\end{equation}
where we used the notation $\check l_i = (l_1,\ldots, l_{i-1},l_i-1, l_{i+1},\ldots, l_d)$.  Then we have the decomposition
\begin{equation}
\partial_{ x_{ i}}u = \sqrt{\frac{ \theta k_i}{\gs_i^2}}\sum_{l\in \bbN^d}  \sqrt{l_i}\mathbf{ 1}_{ l_{ i}\geq1} u_{l}\,  \psi_{\check l_i},
\end{equation}
so that
\begin{equation}
\label{eq:uHrsup}
\Vert u\Vert_{H^r_{\theta}}^2+\sum_{i=1}^d\Vert \partial_{x_i} u\Vert_{H^r_{\theta}}^2 = \sum_{ l\in \mathbb{ N}^{ d}} u_{ l}^{ 2} \left((a_\theta+\lambda_{ l})^{ r} + \sum_{ i=1}^{ d} l_{ i} \mathbf{ 1}_{ l_{ i}\geq1}\frac{ \theta k_{ i}}{ \sigma_{ i}^{ 2}} \left(a_\theta+ \lambda_{l} - \theta k_{ i}\right)^{ r}\right).
\end{equation}
Let us prove the upper bound in \eqref{eq:equiv_norms}: note that for $l_{ i}\geq1$, we have $ \lambda_{ l}\geq \theta k_{ i}$. Hence, since for all $\mu\geq 0$, $r\geq0$, $ \lambda\geq \mu$, $ \left(a_\theta+ \lambda-\mu\right)^{ r}\geq a_{ \theta}^{ r}\frac{ \left(a_\theta+ \lambda\right)^{ r}}{ \left(a_\theta+ \mu\right)^{ r}}$, we deduce that
\begin{align*}
\Vert u\Vert_{H^r_{\theta}}^2+\sum_{i=1}^d\Vert \partial_{x_i} u\Vert_{H^r_{\theta}}^2 &\geq  \sum_{ l\in \mathbb{ N}^{ d}} u_{ l}^{ 2} (a_\theta+\lambda_{ l})^{ r}\left(1 + \sum_{ i=1}^{ d} l_{ i} \mathbf{ 1}_{ l_{ i}\geq1}\frac{\theta k_{ i} a_{ \theta}^{ r}}{ \sigma_{ i}^{ 2}\left(a_\theta+ \theta k_{ i}\right)^{ r}}\right),\\
&\geq  \sum_{ l\in \mathbb{ N}^{ d}} u_{ l}^{ 2} (a_\theta+\lambda_{ l})^{ r}\left(1 + \frac{a_{ \theta}^{ r}}{ \sigma_{ \max}^{ 2}\left(a_\theta+ \theta k_{ \max}\right)^{ r}} \lambda_{ l}\right),
%&\geq  \sum_{ l\in \mathbb{ N}^{ d}} u_{ l}^{ 2} (1+\lambda_{ l})^{ r}\left(a_\theta + \frac{ \lambda_{ l}}{ \sigma_{ \max}^{ 2}\left(a_\theta+ \theta k_{ \max}\right)^{ r}} \right),
\end{align*}
so that the upper bound in \eqref{eq:equiv_norms} is true for $C_{ 2}:= \max \left( \frac{ \sigma_{ \max}^{ 2}\left(a_\theta+ \theta k_{ \max}\right)^{ r}}{ a_{ \theta}^{ r}}, a_{ \theta}\right)$. Concerning the lower bound in \eqref{eq:equiv_norms}, we have from \eqref{eq:uHrsup},
%\begin{align*}
%\Vert u\Vert_{H^r_{\theta}}^2+\sum_{i=1}^d\Vert \partial_{x_i} u\Vert_{H^r_{\theta}}^2 \leq \sum_{ l\in \mathbb{ N}^{ d}} u_{ l}^{ 2} (a_\theta+\lambda_{ l})^{ r} \left(1 + \sum_{ i=1}^{ d} l_{ i} \mathbf{ 1}_{ l_{ i}\geq1}\frac{ \theta k_{ i}}{ \sigma_{ i}^{ 2}}\right)
%\end{align*}
\begin{align*}
\Vert u\Vert_{H^r_{\theta}}^2+\sum_{i=1}^d\Vert \partial_{x_i} u\Vert_{H^r_{\theta}}^2 \leq \sum_{ l\in \mathbb{ N}^{ d}} u_{ l}^{ 2} (a_\theta+\lambda_{ l})^{ r} \left(1 + \frac{  \lambda_{ l}}{ \sigma_{ \min}^{ 2}}\right),
\end{align*}
so that the upper bound holds for $C_{ 1}:= \frac{ 1}{ \min \left( \sigma_{ \min}^{ 2},a_{ \theta}\right)}$.
%\com{alors je trouve plutôt
%\begin{equation}
%\Vert u\Vert_{H^r_{\theta}}^2+\sum_{i=1}^d\Vert \partial_{x_i} u\Vert_{H^r_{\theta}}^2 = \sum_{l\in \bbN^d} \left((1+\lambda_l)^r+\sum_{i=1}^d \frac{ \theta k_{ i}}{ \sigma_{ i}^{ 2}}(l_i+1)(1+\lambda_{\hat l_i})^r\right) u^2_{ l},
%\end{equation}
%}
%which concludes the proof since 
%$\lambda_l \leq\lambda_{\check l_i}\leq \lambda_l$ for all $l\in \bbN^d\setminus \{0\}$.
%$\lambda_l \leq\lambda_{\hat l_i}\leq 2\frac{\max_{j=1\ldots d} k_j}{\min_{j=1\ldots d} k_j} \lambda_l$ for all $l\in \bbN^d\setminus \{0\}$.
%\com{je trouve plutôt
%$\lambda_l \leq\lambda_{\hat l_i}\leq \left(1+\frac{\max_{j=1\ldots d} k_j}{\min_{j=1\ldots d} k_j} \right)\lambda_l$ for all $l\in \bbN^d\setminus \{0\}$.}
\end{proof}

For all $ \theta>0$, the operator $ -\mathcal{ L}^{ \ast}_{ \theta}$ (recall its definition \eqref{eq:L_OU} and its decomposition \eqref{eq:decomp Ltheta}) is sectorial in $L^{ 2}_{ \theta}$ and generates a semigroup $e^{ t \mathcal{ L}^{ \ast}_{ \theta}}$ satisfying (see e.g. \cite{Henry:1981})
 for all $ \alpha\geq 0$, $r\geq0$, and $ \lambda< \theta\min(k_1,\ldots, k_d)$, there exists some $C>0$ such that for all $f\in H^r_{ \theta}$,
\begin{equation}\label{eq:bound epx L}
\left\Vert e^{t\cL^{ \ast}_{ \theta}} f\right\Vert_{H^{r+ \alpha}_{ \theta}} \leq C \left(1+t^{- \alpha/2} e^{ - \lambda t}\right)\Vert f\Vert_{H^r_{ \theta}},
\end{equation}
and for all $f\in H^r_{ \theta}$ such that $\int f w_{\theta}=0$,
\begin{equation}\label{eq:bound epx L with u mean 0}
\left\Vert  e^{t\cL_{ \theta}^{ \ast}}  f\right\Vert_{H^{r+\ga}_{ \theta}} \leq C t^{- \frac{1}{2}}e^{- \lambda t}\Vert f\Vert_{H^r_{ \theta}}.
\end{equation}

\medskip
Let $ \theta^{ \prime}>0$. The point of the following result is to state similar contraction results for $ \mathcal{ L}^{ \ast}_{ \theta^{ \prime}}$ in $H_{ \theta}^{ r}$, in the case $ \theta^{ \prime}\neq \theta$:

\begin{proposition}
\label{prop:control_semigroup_Last}
For all $ 0< \theta\leq\theta^{ \prime}$ the following is true: the operator $ \mathcal{ L}^{ \ast}_{ \theta^{ \prime}}$ generates an analytic semigroup in $H_{ \theta}^{ r}$ and for all $r\geq 0$, $ \alpha\geq0$ and $\lambda < \theta\min(k_1,\ldots,k_d)$, there exists a constant $C>0$ such that for all $f\in H_{ \theta}^{ r}$ and $t>0$
\begin{equation}\label{eq:bound_semigroup_Last}
\left\Vert e^{t\cL^{ \ast}_{ \theta^{ \prime}}} f\right\Vert_{H^{r+ \alpha}_{ \theta}} \leq C \left(1+t^{- \alpha/2} e^{ - \lambda t}\right)\Vert f\Vert_{H^r_{ \theta}},
\end{equation}
and for all $r\geq 1$,
\begin{equation}\label{eq:bound_semigroup_Last_mean0}
\left\Vert \nabla e^{t\cL^{ \ast}_{ \theta^{ \prime}}}  f\right\Vert_{H^{r}_{ \theta}} \leq C t^{-\frac12} e^{- \lambda t}\Vert f\Vert_{H^r_{ \theta}}.
\end{equation}
Moreover for all $r\geq 0$, $0< \varepsilon \leq  1$ and $s\geq0$,
\begin{equation}
\label{eq:diff_semigroup_L*}
\left\Vert \left(e^{ (t+s) \mathcal{ L}^*_{ \theta^{ \prime}}}- e^{ t \mathcal{ L}^*_{ \theta^{ \prime}}}\right)f \right\Vert_{ H_{ \theta}^{ r+1}}\leq C s ^{ \varepsilon}t^{- \frac{ 1}{ 2}-\varepsilon}  e^{ - \lambda t}\left\Vert f \right\Vert_{ H_{ \theta}^{ r}}.
\end{equation}
Finally, there exists $r_0>0$ such that for all $r>r_0$, $t>0$ and all $f\in H^r_\theta$,
\begin{equation}
\label{eq:bound semi group last mean0 2}
\left\Vert e^{t\cL^*_{\theta'}}f-\frac{\int_{\bbR^d} e^{t\cL^*_{\theta'}}fw_\theta}{\int_{\bbR^d}w_\theta}\right\Vert_{H^r_\theta}\leq e^{-\lambda t}\left\Vert  f-\int_{\bbR^d} fw_\theta\right\Vert_{H^r_\theta}.
\end{equation}
\end{proposition}

\begin{proof}[Proof of Proposition~\ref{prop:control_semigroup_Last}]
First remark that for all smooth test function $u$
\begin{equation}
\label{eq:comp_Ls}
\left(\cL_{ \theta}^{ \ast} - \mathcal{ L}^{ \ast}_{ \theta^{ \prime}} \right)u= ( \theta^{ \prime}- \theta) Kx \cdot\nabla u.
\end{equation}
Recalling the decomposition \eqref{eq:decomp Ltheta},since $h'_n(x)= \sqrt{n}h_{n-1}(x)$ and $xh_{n-1}(x)=\sqrt{n}h_n(x)+\sqrt{n-1}h_{n-2}(x)$ (see e.g. \cite{Beals2016}, p.102), we get, for all $l\in \bbN$, 
\begin{align}
\left(\cL^*_\theta-\cL^*_{ \theta^{ \prime}}\right) \psi_{\theta,l} &= ( \theta^{ \prime}-\theta) \sum_{i=1}^d  k_i \sqrt{\frac{\theta k_i}{\gs_i^2}}\sqrt{l_i} x_i \psi_{\theta, l_{ \downarrow_{ i}}}\\
&=( \theta^{ \prime}-\theta) \sum_{i=1}^d  k_i \left(l_i\psi_{\theta,l} +\sqrt{l_i(l_i-1)}\psi_{\theta, l_{ \downdownarrows_{ i}}}\right),
\end{align}
where we used the following notations for the shifts w.r.t. the $i$th coordinate
\begin{align}
l_{ \downarrow_{ i}} &= (l_1,\ldots,l_{i-1}, l_i-1,l_{i+1},\ldots, l_d),\\
l_{ \uparrow_{ i}} &= (l_1,\ldots,l_{i-1}, l_i+1,l_{i+1},\ldots, l_d)
\end{align} and the convention $\psi_l=0$ if $l_i<0$ for some $i\in\{1,\ldots,d\}$. In particular we have, recalling that $\lambda_{\theta,l} = \theta\sum_{i=1}^d k_i l_i$,
\begin{equation}
-\cL^*_{ \theta^{ \prime}}\psi_{\theta,l} =\frac{ \theta^{ \prime}\lambda_{\theta,l}}{\theta}\psi_{\theta,l} +( \theta^{ \prime}-\theta) \sum_{i=1}^d k_i\sqrt{l_i(l_i-1)}\psi_{\theta,l_{ \downdownarrows_{ i}}},
\end{equation}
So, we deduce that for $f=\sum_{l} f_l\psi_{\theta,l}$, with $f_l \in \bbC$ for all $l$,
\begin{align*}
\left\Vert \left(\cL^*_\theta- \frac{ \theta}{ \theta^{ \prime}} \cL^*\right)f\right\Vert_{H^r_\theta}^2
&= \left(1- \frac{ \theta}{ \theta^{ \prime}}\right)^2\sum_l (a_\theta+\lambda_{\theta,l})^r \left\vert  \sum_{i=1}^d \theta k_i\sqrt{(l_i+1)(l_i+2)}f_{l_{ \upuparrows_{ i}}} \right\vert^2.
%&= \left(1- \frac{ \theta}{ \theta^{ \prime}}\right)^2I(r,f), 
\end{align*}
Setting
\begin{align}
\label{eq:diff_LLtheta}
I(r, f)&:=\sum_l (a_\theta+\lambda_{\theta,l})^r \left(\sum_{ i=1}^{ d} \theta k_{ i} \sqrt{(l_{ i}+1)(l_{ i}+2)}f_{l_{  \upuparrows_{ i}}}\right)^{ 2}.
\end{align}
Using Jensen's inequality, we obtain (recalling that $a_\theta=\theta \Tr(K)$),
\begin{align*}
I(r, f)
&\leq \sum_{ l_{ 1}, \ldots, l_{ d}=0}^{ +\infty} \left(a_\theta+ \lambda_{\theta, l}\right)^{ r} \left(\sum_{i=1}^d \theta k_i (l_i+1)\right)\left(\sum_{i=1}^d \theta k_i (l_i+2) \left\vert f_{l_{ \upuparrows_{ i}}} \right\vert^{ 2}\right)\\
&= \sum_{ l_{ 1}, \ldots, l_{ d}=0}^{ +\infty} \left(a_\theta+ \lambda_{\theta, l}\right)^{ r+1}\left(\sum_{i=1}^d \theta k_i (l_i+2) \left\vert f_{l_{ \upuparrows_{ i}}} \right\vert^{ 2}\right)\\
&= \sum_{ l_{ 1}, \ldots, l_{ d}=2}^{ +\infty} \left(a_\theta+ \lambda_{\theta, l-2}\right)^{ r+1} \left(\sum_{i=1}^d \theta k_i l_i \left\vert f_{l_{ 1}-2,\ldots,l_{ i}, l_{ i+1}-2, l_{ d}-2} \right\vert^{ 2}\right)\\
&= \sum_{i=1}^d\sum_{ l_{ 1}, \ldots, l_{ d}=2}^{ +\infty} \left(a_\theta+ \lambda_{\theta, l-2}\right)^{ r+1}  \theta k_i l_i \left\vert f_{l_{ 1}-2,\ldots,l_{ i}, l_{ i+1}-2, l_{ d}-2} \right\vert^{ 2}\\
&= \sum_{i=1}^d\sum_{ l_{ i}=2}^{ +\infty} \sum_{ \substack{l_{ p}=0\\ p\neq i}}^{ +\infty} \left(a_\theta+ \lambda_{\theta, l_{ \downdownarrows_{ i}}}\right)^{ r+1} \theta k_i l_i \left\vert f_{l} \right\vert^{ 2}.
\end{align*}
Now we use that $ \lambda_{\theta, l_{ \downdownarrows_{ i}}} \leq \lambda_{\theta, l}$ for any $l$ and $i$, so that
\begin{equation}
I(r, f)\leq \sum_{i=1}^d \sum_{ l_{ i}=2}^{ +\infty} \sum_{ \substack{l_{ p}=0\\ p\neq i}}^{ +\infty} \left(a_\theta+ \lambda_{\theta,l}\right)^{ r+1} \theta k_i l_i \left\vert f_{l} \right\vert^{ 2},
\end{equation}
This sum is anyway smaller than 
\begin{align}
I(r, f)&\leq \sum_{i=1}^d \sum_{ l_{ i}=0}^{ +\infty} \sum_{ \substack{l_{ p}=0\\ p\neq i}}^{ +\infty} \left(a_\theta+ \lambda_{\theta,l}\right)^{ r+1} \theta k_i l_i \left\vert f_{l} \right\vert^{ 2}= \sum_{ l_{ 1}=0}^{ +\infty} \ldots \sum_{l_{ d}=0}^{ +\infty} \left(a_\theta+ \lambda_{\theta,l}\right)^{ r+1} \lambda_{ \theta,l}\left\vert f_{l} \right\vert^{ 2}, \nonumber\\
&\leq  \sum_{ l_{ 1}=0}^{ +\infty} \ldots \sum_{l_{ d}=0}^{ +\infty} \left(a_\theta+ \lambda_{\theta,l}\right)^{ r+2}\left\vert f_{l} \right\vert^{ 2}=\left\Vert (a_\theta- \mathcal{ L}_{ \theta}^{ \ast}) f\right\Vert_{ H_{ \theta}^{ r}}^{ 2}. \label{eq:estim_Irf}
\end{align}
Coming back to \eqref{eq:diff_LLtheta}, we obtain
\begin{equation}
\left\Vert \left(\cL^*_\theta- \frac{ \theta}{ \theta^{ \prime}} \cL^*_{ \theta^{ \prime}}\right)f\right\Vert_{H^r_\theta} \leq \left(1- \frac{ \theta}{ \theta^{ \prime}}\right) \left\Vert (a_\theta- \mathcal{ L}_{ \theta}^{ \ast}) f\right\Vert_{ H_{ \theta}^{ r}},
\end{equation}
which implies in particular that
\begin{equation}\label{eq:bound L - L theta}
\left \Vert \left(a_\theta- \frac{ \theta}{ \theta^{ \prime}} \cL^*_{ \theta^{ \prime}} \right)f\right\Vert_{H_\theta^r}\leq \left(2- \frac{ \theta}{ \theta^{ \prime}}\right) \left \Vert (a_\theta-\cL^*_\theta) f \right\Vert_{H_\theta^r}= \left(2- \frac{ \theta}{ \theta^{ \prime}}\right)\left\Vert f\right\Vert_{H^{r+2}_\theta}.
\end{equation}
Let us now look for a lower bound. Using $(a+b)^2\geq \frac{\gep}{1+\gep}a^2-\gep b^2$ ($ \varepsilon>0$), we get
\begin{align}
\left\Vert \left(a_\theta - \frac{ \theta}{ \theta^{ \prime}}\cL^*_{ \theta^{ \prime}}\right)f\right\Vert_{H^r_\theta}^2
&= \sum_l (a_\theta+\lambda_{\theta,l})^r \Bigg((a_\theta+\lambda_{\theta, l}) f_{ l}\nonumber\\
&\qquad\qquad\qquad\qquad+ \left(1- \frac{ \theta}{ \theta^{ \prime}}\right) \sum_{ i=1}^{ d} \theta k_{ i} \sqrt{ (l_{ i}+1)(l_{ i}+2)}f_{ l_{ \upuparrows_{ i}}}\Bigg)^{ 2}\nonumber\\
&\geq \frac{ \varepsilon}{1+\varepsilon} \sum_l (a_\theta+\lambda_{\theta,l})^{ r+2}  f_{ l}^{ 2}  - \varepsilon  \left(1- \frac{ \theta}{ \theta^{ \prime}}\right)^{ 2}I(r,f). \label{eq:minor_Ltheta}
\end{align}
Hence, recalling \eqref{eq:estim_Irf}, we obtain
\begin{align}
\left\Vert \left(a_\theta - \frac{ \theta}{ \theta^{ \prime}}\cL^*_{ \theta^{ \prime}}\right)f\right\Vert_{H^r_\theta}^2
&\geq  \varepsilon\left(\frac{1}{1+\varepsilon}  - \left(1- \frac{ \theta}{ \theta^{ \prime}}\right)^{ 2}\right)\left\Vert (a_\theta- \mathcal{ L}_{ \theta}^{ \ast}) f\right\Vert_{ H_{ \theta}^{ r}}^{ 2}
\end{align}
So for $\gep>0$ small enough (depending only on $\theta, \theta^{ \prime}$), there exists a constant $c_{ \theta, \theta^{ \prime}}>0$ such that we have $\left\Vert \left(a_\theta- \frac{ \theta}{ \theta^{ \prime}}\cL^*_{ \theta^{ \prime}}\right)f\right\Vert_{H^r_\theta}\geq c_{ \theta, \theta^{ \prime}}\left\Vert (a_\theta-\cL^*_\theta )f\right\Vert_{H^r_\theta}$. This
means, together with \eqref{eq:bound L - L theta}, that
\begin{equation}
c_{ \theta, \theta^{ \prime}}\left\Vert f\right\Vert_{H^{r+2}_\theta}\leq  \left\Vert \left(a_\theta- \frac{ \theta}{ \theta^{ \prime}}\cL^*_{ \theta^{ \prime}}\right)f\right\Vert_{H^r_\theta} \leq \left(2- \frac{ \theta}{ \theta^{ \prime}}\right)\left\Vert f\right\Vert_{H^{r+2}_\theta}.
\end{equation}
In particular $0$ is in the resolvent set of $a_\theta- \frac{ \theta}{ \theta^{ \prime}}\cL^*_{ \theta^{ \prime}}$, and $a_\theta- \frac{ \theta}{ \theta^{ \prime}}\cL^*_{ \theta^{ \prime}}$ has a compact resolvent, since it is the case for $a_\theta-\cL^{ \ast}_{ \theta}$. So $\cL^*_{ \theta^{ \prime}}$ has a discrete spectrum, composed of a sequence of eigenvalues with modulus going to infinity. But any eigenfunction $\psi$ of $\cL^*_{ \theta^{ \prime}}$ in $H^r_\theta$ is also an eigenfunction of $\cL^*_{ \theta^{ \prime}}$ in $H^r_{ \theta^{ \prime}}$, and thus the eigenvalues of $\cL^*_{ \theta^{ \prime}}$ in $H^r_\theta$ are the $\lambda_{l, \theta^{ \prime}}$'s, with associated eigenfunctions the $\psi_{l, \theta^{ \prime}}$'s. So in particular $\cL^{ \ast}_{ \theta^{ \prime}}$ is sectorial, and thus generates an analytic semigroup $e^{ t \mathcal{ L}^{ \ast}_{ \theta^{ \prime}}}$ in $H^r_{ \theta}$.

\medskip

Let us now prove that the interpolation spaces induced by $\cL^*_{ \theta^{ \prime}}$ and $\cL^*_\theta$ in $H^r_\theta$ are equivalent. Since the operator $ \left(\frac{ \theta}{ \theta^{ \prime}}\mathcal{ L}^{ \ast}_{ \theta^{ \prime}}-\mathcal{ L}^{ \ast}_{ \theta}\right) \left(a_\theta- \frac{ \theta}{ \theta^{ \prime}}\cL^*_{ \theta^{ \prime}}\right)^{-1}$ (and thus $(1+(\theta\mathcal{ L}^{ \ast}- \mathcal{ L}^{ \ast}_{ \theta})( a_\theta-\theta\cL^*)^{-1})^\ga$ for $\ga \geq 0$) is bounded in $H^r_\theta$, we obtain
\begin{align}
\Vert  f\Vert_{H^{r+\ga}_\theta}&=  \left\Vert \left( a_\theta-\cL^*_\theta\right)^{\ga/2} f \right\Vert_{H^r_\theta} \nonumber\\
&= \left\Vert  \left(1+ \left( \frac{ \theta}{ \theta^{ \prime}}\cL^*_{ \theta^{ \prime}}- \cL^*_\theta\right) \left(a_\theta- \frac{ \theta}{ \theta^{ \prime}}\cL^*_{ \theta^{ \prime}}\right)^{-1}\right)^{\ga/2} \left( a_\theta- \frac{ \theta}{ \theta^{ \prime}}\cL^*_{ \theta^{ \prime}}\right)^{\ga/2} f \right\Vert_{H^r_\theta}\\
& \leq C \left\Vert \left( a_\theta- \frac{ \theta}{ \theta^{ \prime}}\cL^*_{ \theta^{ \prime}}\right)^{\ga/2} f \right\Vert_{H^r_\theta} \nonumber.
\end{align}
The inverse bound $ \left\Vert \left(a_\theta- \frac{ \theta}{ \theta^{ \prime}}\cL^*_{ \theta^{ \prime}}\right)^{\ga/2} f \right\Vert_{H^r_\theta} \leq C \Vert  f\Vert_{H^{r+\ga}_\theta}$ follows from similar arguments.

\medskip

We are now in condition to prove \eqref{eq:bound epx L}, applying \cite{Henry:1981}, Th.~1.4.3. Indeed, $\cL^*_{\theta'}$ has a real spectrum smaller than $-\theta'\min(k_1,\ldots,k_d)$ on the the subspace of $H^r_{\theta}$ generated by the eigenfunctions $\psi_{l,\theta'}$ with $l\neq 0$, so applying this Theorem we get, denoting $\cP_{\theta'}f=f-\frac{\int_{\bbR^d} f w_{\theta'}}{\left(\int_{\bbR^d}w_{\theta'}\right)^2}$ the projection on this subspace (which is an element of $\mathcal{B}(H^r_\theta)$ for $0<\theta\leq \theta'$),
\begin{equation}
\left\Vert e^{t\cL^*_{\theta'}}\cP_{\theta'}f\right\Vert_{H^{r+\ga}_{\theta}}
\leq C_1 \left\Vert \left( a_\theta- \frac{ \theta}{ \theta^{ \prime}}\cL^*_{ \theta^{ \prime}}\right)^{\ga/2}e^{t\cL^*_{\theta'}}\cP_{\theta'}f\right\Vert_{H^{r}_{\theta}}
\leq C_2 t^{-\ga/2}e^{-\lambda t}\left\Vert f\right\Vert_{H^{r}_{\theta}}.
\end{equation}
This implies $\eqref{eq:bound epx L}$, since $e^{t\cL^*_{\theta'}}\left( 1-\cP_{\theta'}\right)f=\left( 1-\cP_{\theta'}\right)f$.

\medskip

The proof of \eqref{eq:bound epx L with u mean 0} relies on the classical identity, valid for $f\in L^2_{\theta'}$,
\begin{equation}
e^{t\cL^*_{\theta'}} f = \bbE\left[f\left(e^{-t\theta' K}x+\sqrt{1-e^{-2t\theta' K} }G_{\theta'}\right)\right],
\end{equation}
where $G_{\theta'}$ a gaussian variable on $\mathbb{R}^d$ with mean $0$ and variance $(\theta' K)^{-1}\gs^2$.
This implies directly, for $f\in H^r_{\theta}$ with $r\geq 1$,
\begin{equation}
\nabla e^{t\cL^*_N} f = e^{-t\theta' K}\,  e^{t\cL^*_N} \nabla f,
\end{equation}
and thus, denoting $\underline{K}=\min\{k_1,\ldots,k_d\}$,
\begin{multline}
\left\Vert \nabla e^{t\cL^*_N}f\right\Vert_{H^{r}_\theta}\leq e^{-\theta' \underline{K} t}\left\Vert e^{t\cL^*_N} \nabla f\right\Vert_{H^{r}_\theta}\leq C\left(1+t^{-\frac12}e^{-\lambda t}\right) e^{-\theta' \underline{K} t}\left\Vert \nabla f\right\Vert_{H^{r-1}_\theta}\\
\leq C't^{-\frac12}e^{-\lambda t}\left\Vert  f\right\Vert_{H^{r}_\theta}.
\end{multline}

\medskip
For the proof of the third assertion, since \cite{Henry:1981}, Th.~1.4.3. implies
that for $0< \gep\leq 1$,
\begin{equation}
\left\Vert\left(e^{s\mathcal{ L}^*_{ \theta^{ \prime}}}-1\right) f\right\Vert_{ H_{ \theta}^{ r}}\leq C_\gep s^{\gep}\left\Vert f\right\Vert_{ H_{ \theta}^{ r+2\gep}},
\end{equation}
we obtain, since $\left(e^{s\mathcal{ L}^*_{ \theta^{ \prime}}}-1\right)\cP_{\theta'}f=\left(e^{s\mathcal{ L}^*_{ \theta^{ \prime}}}-1\right)f$ and $\cP_{\theta'}$ commutes with $ e^{ t \mathcal{ L}^*_{ \theta^{ \prime}}}$,
\begin{multline}
\left\Vert \left(e^{ (t+s) \mathcal{ L}^*_{ \theta^{ \prime}}}- e^{ t \mathcal{ L}^*_{ \theta^{ \prime}}}\right)f\right\Vert_{ H_{ \theta}^{ r+1}} =\left\Vert \left(e^{s\mathcal{ L}^*_{ \theta^{ \prime}}}-1\right)e^{ t \mathcal{ L}^*_{ \theta^{ \prime}}}\cP_{\theta'} f\right\Vert_{ H_{ \theta}^{ r+1}}
\leq C_{\gep}s^\gep\left\Vert e^{ t \mathcal{ L}^*_{ \theta^{ \prime}}}\cP_{\theta'} f\right\Vert_{ H_{ \theta}^{ r+1+2\gep}}
\\
\leq Cs^\gep t^{-\frac12-\gep}e^{-\lambda t}\left\Vert f\right\Vert_{ H_{ \theta}^{ r}}.
\end{multline}

\medskip

The last assertion is not a direct consequence of the estimates obtained above, since the hypothesis $\int_{\bbR^d} f w_\theta=0$ is not well adapted to the eigenfunctions $\psi_{l,\theta'}$ of $\cL^*_{\theta'}$. In particular, having $\int_{\bbR^d}f w_\theta =0$ does not imply $\int_{\bbR^d} e^{t\cL^*_{\theta'}} f w_\theta =0$, while it is the case when $\theta=\theta'$. We will only be able to obtain this estimates for $r$ large enough, via direct calculations. Remark first that
\begin{multline}
\frac12\frac{d}{dt}\left\Vert e^{t\cL^*_{\theta'}}f-\frac{\int_{\bbR^d} e^{t\cL^*_{\theta'}}fw_\theta}{\int_{\bbR^d}w_\theta}\right\Vert_{H^r_\theta}^2\\=\left\langle \cL^*_{\theta'} e^{t\cL^*_{\theta'}}f-\frac{\int_{\bbR^d} \cL^*_{\theta'} e^{t\cL^*_{\theta'}}fw_\theta}{\int_{\bbR^d}w_\theta},e^{t\cL^*_{\theta'}}f-\frac{\int_{\bbR^d} e^{t\cL^*_{\theta'}}fw_\theta}{\int_{\bbR^d}w_\theta}\right\rangle_{H^r_\theta}.
\end{multline}
Recalling that $\cL^*_{\theta'}a=0$ and remarking that $\left\langle a, e^{t\cL^*_{\theta'}}f-\frac{\int_{\bbR^d} e^{t\cL^*_{\theta'}}fw_\theta}{\int_{\bbR^d}w_\theta}\right\rangle_{H^r_\theta}=0$  for any constant $a$, we get
\begin{multline}
\frac12\frac{d}{dt}\left\Vert e^{t\cL^*_{\theta'}}f-\frac{\int_{\bbR^d} e^{t\cL^*_{\theta'}}fw_\theta}{\int_{\bbR^d}w_\theta}\right\Vert_{H^r_\theta}^2\\=\left\langle \cL^*_{\theta'}\left(e^{t\cL^*_{\theta'}}f-\frac{\int_{\bbR^d} e^{t\cL^*_{\theta'}}fw_\theta}{\int_{\bbR^d}w_\theta}\right),e^{t\cL^*_{\theta'}}f-\frac{\int_{\bbR^d} e^{t\cL^*_{\theta'}}fw_\theta}{\int_{\bbR^d}w_\theta}\right\rangle_{H^r_\theta},
\end{multline}
so the proof of the last assertion reduces to the study of $\langle \cL^*_{\theta'} f,f\rangle_{H^r_\theta}$ with $\int_{\bbR^d} f w_\theta=0$. Now for $f$ satisfying $\int_{\bbR^d} f w_\theta=0$, with decomposition $f=\sum_{l\neq 0} f_l \psi_{l,\theta}$, we get
\begin{align}
- \frac{ \theta}{ \theta^{ \prime}} \left \langle  \cL^*_{ \theta^{ \prime}}f, f \right\rangle_{H^r_\theta}& =\sum_{l\neq 0} (a_\theta+\lambda_{\theta,l})^r\lambda_{\theta,l} \vert f_l\vert^2  \nonumber\\ 
& \qquad +  \left(1- \frac{ \theta}{ \theta^{ \prime}}\right)\left( \sum_{l\neq 0}(a_\theta+\lambda_{\theta,l})^r\sum_{i=1}^d \theta k_i \sqrt{(l_i+1)(l_i+2)}f_l \bar f_{l_{ \upuparrows_{ i}}}\right). \label{aux:ReLast}
\end{align}
Now remark that for the second term, using Cauchy Schwarz inequality, we get
\begin{align*}
&\left|\sum_{l\neq 0} (a_\theta+\lambda_{\theta,l})^r   \sum_{i=1}^d \theta k_i \sqrt{(l_i+1)(l_i+2)}f_l \bar f_{l_{ \upuparrows_{ i}}} \right| \\
&=\left|\sum_{l\neq 0} \left\lbrace(a_\theta+\lambda_{\theta,l})^{ r/2}\sqrt{\lambda_{ \theta,l}}f_{ l} \right\rbrace \left\lbrace \frac{ \left(a_\theta+ \lambda_{ \theta,l}\right)^{ r/2}}{ \sqrt{\lambda_{\theta, l}}}\sum_{i=1}^d \theta k_i \sqrt{(l_i+1)(l_i+2)} \bar f_{l_{ \upuparrows_{ i}}}\right\rbrace\right| \\
&\leq \left|\sum_{l\neq 0} (a_\theta+\lambda_{\theta,l})^{ r}\lambda_{\theta,l} |f_l|^2\right|^{\frac12}\left|\sum_{l\neq 0} \frac{\left(a_\theta+ \lambda_{\theta, l}\right)^{ r}}{\lambda_{\theta,l}} \left(\sum_{i=1}^d \theta k_i \sqrt{(l_i+1)(l_i+2)} \bar f_{ l_{ \upuparrows_{ i}}}\right)^{ 2} \right|^{\frac12}.
\end{align*}
Using Jensen's inequality
\begin{multline*}
\left|\sum_{l\neq 0} \frac{\left(a_\theta+ \lambda_{\theta, l}\right)^{ r}}{\lambda_{\theta,l}} \left(\sum_{i=1}^d \theta k_i \sqrt{(l_i+1)(l_i+2)} \bar f_{l_{ \upuparrows_{ i}}}\right)^{ 2} \right| \\\leq \sum_{l\neq 0} \frac{\left(a_\theta+ \lambda_{\theta, l}\right)^{ r}}{\lambda_{\theta,l}} \left(\sum_{ i=1}^{ d} \theta k_{ i}(l_{ i}+2)\right)\sum_{i=1}^d \theta k_i (l_i+2) \left| f_{l_{ \upuparrows_{ i}}}\right|^2,\\
=\sum_{l\neq 0} \frac{\left(a_\theta+ \lambda_{\theta, l}\right)^{ r}}{\lambda_{\theta,l}} \left(2 a_{ \theta} + \lambda_{ \theta, l}\right)\sum_{i=1}^d \theta k_i (l_i+2) \left| f_{l_{ \upuparrows_{ i}}}\right|^2,
\end{multline*}
Denoting by $\mathcal{N}_{i}:= \upuparrows_{ i} \left( \mathbb{ N}^{ d}\setminus \left\lbrace 0\right\rbrace\right)= \left\lbrace l\in \mathbb{ N}^{ d}, l_{ i}\geq2, \sum_{ j=1}^{ d}l_{ j}\geq3\right\rbrace$, we obtain
\begin{align*}
&\left|\sum_{l\neq 0} \frac{\left(a_\theta+ \lambda_{\theta, l}\right)^{ r}}{\lambda_{\theta,l}} \left(\sum_{i=1}^d \theta k_i \sqrt{(l_i+1)(l_i+2)} \bar f_{l_{ \upuparrows_{ i}}}\right)^{ 2} \right| \\
&\leq 2\sum_{i=1}^d \sum_{l \in \mathcal{ N}_{ i}}\frac{\left(a_\theta+ \lambda_{\theta, l_{ \downdownarrows_{ i}}}\right)^{ r}}{\lambda_{\theta, l_{ \downdownarrows_{ i}}}} \left(2a_\theta+ \lambda_{\theta, l_{ \downdownarrows_{ i}}}\right) \theta k_i l_i \left\vert f_{l} \right\vert^{ 2}\\
&= \sum_{i=1}^d\sum_{l\in \mathcal{ N}_{ i}}\frac{\left(a_\theta+ \lambda_{\theta,l}-2\theta k_i\right)^{ r}}{\lambda_{\theta,l}-2\theta k_i} \left(2a_\theta+ \lambda_{\theta,l}-2\theta k_i\right) \theta k_i l_i \left\vert f_{l} \right\vert^{ 2}\\
&= \sum_{i=1}^d\sum_{l\in \mathcal{ N}_{ i}}b_{\theta,l,i}\left(a_\theta+ \lambda_{\theta,l}\right)^{ r}\theta k_i l_i \left\vert f_{l} \right\vert^{ 2},
\end{align*}
where 
\begin{equation}
b_{\theta,l,i} := \frac{\left(a_\theta+ \lambda_{\theta,l}-2\theta k_i\right)^{ r} \left(2a_\theta+ \lambda_{\theta,l}-2\theta k_i\right)}{\left(a_\theta+\lambda_{\theta,l}\right)^{ r}(\lambda_{\theta,l}-2\theta k_i)} = \left(1-\frac{2\theta k_i}{a_\theta+\lambda_{\theta,l}} \right)^r \left(1+\frac{2a_\theta}{\lambda_{\theta,l}-2\theta k_i}\right).
\end{equation}
Now, for $l\in \mathcal{ N}_{ i}$, we have
\begin{equation}
a_\theta +\lambda_{\theta,l} \leq d \theta  k_{max}+\lambda_{\theta,l} \leq (d+1)\frac{k_{max}}{k_{min}}\lambda_{\theta,l}, 
\end{equation}
and
\begin{equation}
\lambda_{\theta,l}-2\theta k_i \geq \theta \left(\sum_{j=1}^d l_j -2\right) k_{min}\geq \frac{k_{min}}{3k_{max}}\lambda_{\theta,l},
\end{equation}
so that
\begin{equation}
\label{eq:br}
b_{\theta,l,i}\leq  \left(1-\frac{2 k_{min}^2}{(d+1)k_{max}}\cdot \frac{1}{\sum_{j=1}^dk_jl_j} \right)^r \left(1+\frac{6 d\, k^2_{max}}{k_{min}}\frac{1}{\sum_{j=1}^dk_jl_j}\right).
\end{equation}
Now, observe that for $c_{ 1}, c_{ 2}>0$, $x \mapsto \left(1- \frac{ c_{ 1}}{ x}\right)^{ r} \left(1+ \frac{ c_{ 2}}{ x}\right)$ is strictly increasing with limit $1$ as $x\to\infty$, provided that $r> \frac{ c_{ 2}}{ c_{ 1}}$.
Hence, taking $r$ large enough in \eqref{eq:br} ($r$ depending only on $K$ and $d$, not on $l$, $i$ nor $ \theta$) we have $|b_{\theta,l,i}| \leq1$, which means that the second term of the right-hand side of \eqref{aux:ReLast} is bounded as follows:
\begin{equation}
\left|\sum_{l\neq 0}(a_\theta+\lambda_{\theta,l})^r\sum_{i=1}^d \theta k_i \sqrt{(l_i+1)(l_i+2)}f_l \bar f_{l_{ \upuparrows_{ i}}}\right|\leq \sum_{l\neq 0} (a_\theta+\lambda_{\theta,l})^{ r}\lambda_{\theta,l} |f_l|^2.
\end{equation}
We deduce from \eqref{aux:ReLast} and this estimate that
\begin{equation}
\left|- \frac{ \theta}{ \theta^{ \prime}} \left \langle \mathcal{ P}^\perp_\theta \cL^*_{ \theta^{ \prime}} \mathcal{ P}^\perp_\theta f,\bar f \right\rangle_{H^{r}_\theta}-\sum_{l\neq 0} (a_\theta+\lambda_{\theta,l})^{ r}\lambda_{\theta,l} |f_l|^2\right|\leq \left(1- \frac{ \theta}{ \theta^{ \prime}}\right)\sum_{l\neq 0} (a_\theta+\lambda_{\theta,l})^{ r}\lambda_{\theta,l} |f_l|^2,
\end{equation}
which means that
\begin{equation}
-\mathrm{Re}\left \langle  \cL^*_{ \theta^{ \prime}} f,\bar f \right\rangle_{H^r_\theta} \geq \sum_{l\neq 0} (a_\theta+\lambda_{\theta,l})^r\lambda_{\theta,l} \vert f_l\vert^2 \geq \theta k_{min} \Vert f\Vert_{H^r_\theta}^2. 
\end{equation}
\end{proof}

As already stated in Section~\ref{sec:sobolev}, we rely in this paper on a ``pivot" space structure (see \cite{MR697382}, pp.81-82): observe first that for $u \in L_{ -\theta}^{ 2}$, $  v\in L_{ \theta}^{ 2} \mapsto \int_{ \mathbb{ R}^{ d}} uv {\rm d}x$ defines a continuous linear form on $L_{ \theta}^{ 2}$. Respectively, for $u\in \left(L_{ \theta}^{ 2}\right)^{ \prime}$, the mapping $ \psi \mapsto Tu(\psi):= \left\langle u\, ,\, \psi w_{ - \theta}\right\rangle$ defines a continuous linear form on $L^{ 2}$  (that is the usual $L^{ 2}$ space without weight, i.e. $w\equiv 1$ in \eqref{eq:norm_Lw}). By Riesz Theorem, there exists $v\in L^{ 2}$, such that $Tu(\psi)= \int v \psi=\int \tilde{ v} \tilde{\psi}$, $ \psi\in L^{ 2}$, for $ \tilde{ v}:=v w_{ \theta/2} \in L^{ 2}_{ - \theta}$, $ \tilde{ \psi}= \psi w_{ - \theta/2}\in L^{ 2}_{  \theta}$. This observation permits the identification of $(L^2_\theta)'$ with $L^2_{-\theta}$ (and hence, $\langle\cdot,\cdot\rangle_{(L^2_\theta)'\times L^2_\theta}$ with $\langle\cdot,\cdot\rangle_{L^2}=\langle\cdot,\cdot\rangle$). Now, since $H^r_\theta\rightarrow L^2_\theta$ is dense, we have a dense injection $(L^2_\theta)'\rightarrow H^{-r}_\theta$. With the identification $(L^2_\theta)'\approx L^2_{-\theta}$, we obtain, for all $u\in L^2_{-\theta}\subset H^{-r}_\theta$ and all $f\in H^r_\theta$,
\[
\langle u, f\rangle_{H^{-r}_\theta\times H^r_\theta} = \langle u,f\rangle.
\]
Remark in particular that if $u\in L^2_{-\theta}$, then for all $f\in H^{r+1}_\theta$ we have
\begin{equation}
\left|\langle \partial_{x_i} u ,f\rangle\right|=\left|-\langle u ,\partial_{x_i} f\rangle\right|\leq C \Vert u\Vert_{H^{-r}_\theta}\Vert f\Vert_{H^{r+1}_\theta},
\end{equation}
so that if $u\in H^{-r}_\theta$ then $\nabla u\in H^{-(r+1)}_\theta$ with
\begin{equation}\label{eq: nabla u H'}
\Vert \nabla u\Vert_{ H^{-(r+1)}_\theta}\leq C \Vert u\Vert_{H^{-r}_\theta}.
\end{equation}
With this structure, since $L^2_\theta$ is reflexive the closure of $\cL_{ \theta^{ \prime}}$, seen as an operator on $(L^2_\theta)'$, is the adjoint of $\cL^{ \ast}_{ \theta^{ \prime}}$ (\cite{Kato1995}, Th. 5.29) and is thus sectorial and defines an analytical semi-group $e^{t\cL_{\theta'}}$ in $H^{-r}_\theta$. In the same way, since $H^r_\theta$ is reflexive, the adjoint of $e^{t\cL_{ \theta^{ \prime}}^{ \ast}}$ seen as an operator on $H^r_\theta$ is $e^{t\cL_{ \theta^{ \prime}}}$ seen as an operator on $H^{-r}_\theta$ (\cite{Pazy1983}, Cor.~10.6).

From Proposition~\ref{prop:control_semigroup_Last} and the structure described above we deduce directly the following estimates for the semi-group induced by $\cL_{ \theta^{ \prime}}$ (recall \eqref{eq:def Ltheta}) in $ H^{-r}_\theta$ and $t>0$.

\begin{proposition}
\label{prop:control_semigroup}
For all $0< \theta\leq \theta^{ \prime}$ the operator $\cL_{\theta'}$ is sectorial and generates an analytical semi-group in $H^{-r}_\theta$. Moreover we have the following estimates: for any $r\geq 0$, $ \alpha\geq0$ and $\lambda < \theta \min(k_1,\ldots,k_d)$ there exists a constant $C>0$ such that for all $u\in H_{ \theta}^{ -(r+\ga)}$,
\begin{equation}\label{eq:bound_semigroup2}
\left\Vert e^{t\cL_{ \theta^{ \prime}}} u\right\Vert_{H^{-r}_{ \theta}} \leq C \left(1+t^{- \alpha/2} e^{ - \lambda t}\right)\Vert u\Vert_{H^{-(r+\ga)}_{ \theta}},
\end{equation}
and for all $r\geq 1$,
\begin{equation}\label{eq:bound_semigroup_mean02}
\left\Vert e^{t\cL_{ \theta^{ \prime}}} \nabla u\right\Vert_{H^{-r}_{ \theta}} \leq C t^{- \frac12} e^{- \lambda t}\Vert u\Vert_{H^{-r}_{ \theta}}\, .
\end{equation}
Moreover for all $r\geq 0$, $ 0<\varepsilon\leq 1$ and $s\geq0$,
\begin{equation}
\label{eq:diff_semigroup_L2}
\left\Vert \left(e^{ (t+s) \mathcal{ L}_{ \theta^{ \prime}}}- e^{ t \mathcal{ L}_{ \theta^{ \prime}}}\right)u \right\Vert_{ H_{ \theta}^{ -r}}\leq C s ^{ \varepsilon}t^{-\frac12-\gep} e^{ - \lambda t} \left\Vert u \right\Vert_{ H_{ \theta}^{ -(r+1)}}.
\end{equation}
Finally, there exist $r_0>0$, $C>0$ such that for any $0<\theta\leq\theta'$, for all $r>r_0$, $t>0$ and all $u \in H^{-r}_\theta$ satisfying $\int u=0$,
\begin{equation}
\label{eq:bound semi group mean0 2 2}
\left\Vert e^{t\cL_{\theta'}}u\right\Vert_{H^{-r}_\theta}\leq C e^{-\lambda t}\left\Vert  u\right\Vert_{H^{-r}_\theta}.
\end{equation}
\end{proposition}

\begin{proof}[Proof of Proposition~\ref{prop:control_semigroup}]
The spectral structure of $\cL_{\theta'}$ follows directly from the one of $\cL^*_{\theta'}$. To prove the first estimate of the proposition it is now sufficient to remark that for all $f\in H^r_\theta, u\in L^2_{-\theta}$,
\[
\left|\langle e^{t\cL_{ \theta^{ \prime}}}u,f\rangle\right| =\left|\langle u,e^{t\cL_{ \theta^{ \prime}}^{ \ast}}f\rangle\right|\leq C \left(1+t^{- \alpha/2} e^{ - \lambda t}\right)\Vert f\Vert_{H^r_\theta}\Vert u\Vert_{H^{-(r+\ga)}_{ \theta}}.
\]
For the second point,
\[
\left|\langle e^{t\cL_{ \theta^{ \prime}}}\nabla u,f\rangle\right|=\left|\langle u,\nabla e^{t\cL_{ \theta^{ \prime}}^{ \ast}}f\rangle\right|\leq  C t^{- \alpha/2} e^{- \lambda t}\left\Vert f\right\Vert_{H^r_\theta}\Vert u\Vert_{H^{-(r+\ga)}_{ \theta}}.
\]
The third point follows from similar estimates. For the last point, remark that
if $ \left\langle u\, ,\, 1\right\rangle=0$,
\begin{multline}
\left|\langle e^{t\cL_{ \theta^{ \prime}}}u,f\rangle\right|=\left|\left\langle u,e^{t\cL^*_{ \theta^{ \prime}}}f\right\rangle\right|=\left|\left\langle u,e^{t\cL^*_{ \theta^{ \prime}}}f-\frac{\int_{\bbR^d} e^{t\cL^*_{\theta'}}fw_\theta}{\int_{\bbR^d}w_\theta} \right\rangle\right|\\
\leq  C t^{- \alpha/2} e^{- \lambda t}\left\Vert f-\int fw_\theta\right\Vert_{H^r_\theta}\Vert u\Vert_{H^{-(r+\ga)}_{ \theta}},
\end{multline}
and $\left\Vert f-\int fw_\theta\right\Vert_{H^r_\theta}\leq 2\Vert f\Vert_{H^r_\theta}$.

\end{proof}

\section{Gr\"onwall Lemma}\label{app:Gron}

\begin{lemma}\label{lem:GH}
Let $t \mapsto y_t$ be a nonnegative and continuous function on $[0, T )$
satisfying, for all $t \in [0, T )$ and some positive constants $c_0$ and $c_1$,
\begin{equation}
y_t\leq c_0+c_1\int_0^t \left(1+\frac{1}{\sqrt{t-s}}\right)y_s \dd s.
\end{equation}
Then for all $t \in [0, T )$, $y_t \leq 2c_0 e^{\ga t}$ with $\ga = 2c_1 + 4c_1^2  \left(\Gamma\left(\frac12\right)\right)^2$, where $\Gamma$ is the usual special function$\Gamma(r) =\int_0^\infty x^{r-1} e^{-x} dx$.
\end{lemma}

For the proof of this Lemma, see \cite{Giacomin:2012}, Lemma 5.2.

\begin{lemma}
\label{lem:Gronwall}
Let $a, b, \lambda>0$ and $ \phi$ a nonnegative measurable function on $[0, +\infty)$ such that $ \phi$ is integrable on $[0, +\infty)$. Suppose that $t\geq 0 \mapsto u_{ t}$ is a nonnegative function satisfying
\begin{equation}
u_{ t} \leq a + b \int_{ 0}^{t} \left(1+ \phi(t-s)\right) e^{ - \lambda s} u_{ s}{\rm d}s.
\end{equation}
Then, there exists some constant $C(b, \phi)>0$ such that 
\begin{equation}
\label{eq:Gronwall}
\sup_{ t\geq 0} u_{ t} \leq 2a \exp \left( \frac{ C(b, \phi)}{  \lambda}\right).
\end{equation}
\end{lemma}
\begin{proof}[Proof of Lemma~\ref{lem:Gronwall}]
Define $A= A(b, \phi)\geq0$ such that $ \int_{ 0}^{+\infty} \phi(u) \mathbf{ 1}_{ \left\lbrace \phi(u) \geq A\right\rbrace} {\rm d}u \leq \frac{ 1}{ 2b}$. Then, for all $v\leq t$
\begin{align*}
u_{ v} & \leq a+ b \int_{ 0}^{v} e^{ - \lambda s} u_{ s}{\rm d}s + b \int_{ 0}^{v} \phi(v-s) \mathbf{ 1}_{  \left\lbrace \phi(v-s) \geq A\right\rbrace} e^{ - \lambda s} u_{ s}{\rm d}s \\
&\quad + b \int_{ 0}^{v} \phi(v-s) \mathbf{ 1}_{  \left\lbrace \phi(v-s) \leq A\right\rbrace} e^{ - \lambda s} u_{ s}{\rm d}s,\\
& \leq a+ b \int_{ 0}^{v} e^{ - \lambda s} u_{ s}{\rm d}s + b \sup_{ s\leq v} u_{ s}\int_{ 0}^{v} \phi(v-s) \mathbf{ 1}_{  \left\lbrace \phi(v-s) \geq A\right\rbrace} {\rm d}s + b A\int_{ 0}^{v}  e^{ - \lambda s} u_{ s}{\rm d}s,\\
& \leq a+ b \left(1+A\right) \int_{ 0}^{v} e^{ - \lambda s} u_{ s}{\rm d}s + \frac{ 1}{ 2} u^{ \ast}_{ v} \leq a+ b \left(1+A\right) \int_{ 0}^{t} e^{ - \lambda s} u^{ \ast}_{ s}{\rm d}s + \frac{ 1}{ 2} u^{ \ast}_{ t} ,
\end{align*}
where we have defined $u^{ \ast}(s):= \sup_{ r\leq s} u_{ r}$. Since the last inequality is true for all $v\leq t$, we get
\begin{align*}
u^{ \ast}_{ t} \leq 2a+ 2b \left(1+A\right) \int_{ 0}^{t} e^{ - \lambda s} u^{ \ast}_{ s}{\rm d}s.
\end{align*}
The usual Gr\"onwall lemma applied to $ t \mapsto u^{ \ast}_{ t}$ gives the conclusion, for $C(b, \phi)= 2b(1+ A(b, \phi))$.
\end{proof}

\section*{Acknowledgements}

Both authors benefited from the support of the ANR–19–CE40–0023 (PERISTOCH),
C. Poquet from the ANR-17-CE40-0030 (Entropy, Flows, Inequalities), E. Lu\c con from the ANR-19-CE40-002 (ChaMaNe).

\end{document}